\RequirePackage{tikz}
\documentclass[a4,10pt]{article}

\usepackage{amssymb,amsmath}

\usepackage{amsthm}
\usepackage{bm}
\usepackage{environ} 
\usepackage{todonotes}
\usepackage{booktabs}
\usepackage{csquotes}
\usepackage{mathtools}
\usepackage{subcaption}
\usepackage{picins}
\usepackage{enumitem}
\setenumerate{label=\alph*)}
\usepackage[percent]{overpic}
\usepackage[pdftex,colorlinks=true]{hyperref}
\usepackage{authblk}
\usepackage[margin=1.2in]{geometry}

\theoremstyle{definition}
\newtheorem{thm}{Theorem}[section]
\newtheorem{defn}[thm]{Definition}

\newtheorem{prop}[thm]{Proposition}
\newtheorem{rem}[thm]{Remark}
\newtheorem{cor}[thm]{Corollary}

\renewcommand{\bf}{\mathbf f}

\newcommand{\mbfy}{\mathbf y}

\newcommand{\bh}{\mathbf h}

\newcommand{\bn}{\mathbf n}

\newcommand{\bx}{\mathbf x}
\newcommand{\bs}{\mathbf s}
\newcommand{\bg}{\mathbf g}
\newcommand{\bA}{\mathbf A}
\newcommand{\bC}{\mathbf C}
\newcommand{\bM}{\mathbf M}
\newcommand{\bN}{\mathbf N}
\newcommand{\bB}{\mathbf B}
\newcommand{\bS}{\mathbf S}
\newcommand{\bT}{\mathbf T}

\newcommand{\bD}{\mathbf D}

\newcommand{\bR}{\mathbf Q_{\by^*}}

\newcommand{\bI}{\mathbf I}

\newcommand{\bG}{\mathbf G}
\newcommand{\bJ}{\mathbf J}

\newcommand{\bzero}{\mathbf 0}
\newcommand{\bv}{\mathbf v}

\newcommand{\by}{\mathbf y}

\newcommand{\bw}{\mathbf w}
\newcommand{\bU}{\mathbf U}
\newcommand{\bV}{\mathbf V}
\newcommand{\bu}{\mathbf u}

\newcommand{\bssigma}{\boldsymbol{\sigma}}

\newcommand{\R}{\mathbb R}
\newcommand{\C}{\mathbb C}

\newcommand{\tm}{\subseteq}
\newcommand{\abs}[1]{\lvert #1\rvert}
\newcommand{\norm}[1]{\lVert #1\rVert}

\newcommand{\be}{\mathbf e}

\newcommand{\qta}{\quad\text{ and }\quad}

\renewcommand{\R}{\mathbb{R}}
\newcommand{\N}{\mathbb{N}}

\renewcommand{\Vec}[1]{\renewcommand*{\arraystretch}{1}\begin{pmatrix*}[r]#1\end{pmatrix*}}

\newcommand{\from}{\colon}
\newcommand{\ii}{\mathrm{i}}

\DeclareMathOperator{\diag}{diag}

\DeclareMathOperator{\Span}{span}
\DeclareMathOperator{\im}{Im}

\makeatletter
\newcommand{\vast}{\bBigg@{3}}
\newcommand{\Vast}{\bBigg@{4}}
\makeatother
\newcommand{\vastl}{\mathopen\vast}

\newcommand{\vastr}{\mathclose\vast}

\makeatletter
\newif\ifcenter@asb@\center@asb@false
\def\center@arstrutbox{%
	\setbox\@arstrutbox\hbox{$\vcenter{\box\@arstrutbox}$}%
}
\newcommand*{\CenteredArraystretch}[1]{%
	\ifcenter@asb@\else
	\pretocmd{\@mkpream}{\center@arstrutbox}{}{}%
	\center@asb@true
	\fi
	\renewcommand{\arraystretch}{#1}%
}
\makeatother

\usepackage{tikz}
\usetikzlibrary{patterns}
\definecolor{colorA}{rgb}{0,0.447,0.741}
\definecolor{colorB}{rgb}{0.85,0.325,0.098}
\definecolor{colorE}{rgb}{0.929,0.694,0.125}
\definecolor{colorF}{rgb}{0.494,0.184,0.556}
\definecolor{colorD}{rgb}{0.466,0.674,0.188}
\definecolor{colorC}{rgb}{0.301,0.745,0.933}
\definecolor{colorG}{rgb}{0.635,0.078,0.184}

\newlength{\tickl}    
\tikzset{axes/.style={thick,-latex}}
\tikzset{lineplot/.style={thick}}
\tikzset{arrow/.style={thick,-latex}} 
\tikzset{thick arrow/.style={ultra thick,-latex}} 
\tikzset{grid lines/.style={very thin,gray!30}}	
\tikzset{point/.style={radius=2pt}}
\tikzset{help line/.style={black,thin,dashed}} 

\makeatletter
\newsavebox{\measure@tikzpicture}
\NewEnviron{scaletikzpicturetowidth}[1]{%
\def\tikz@width{#1}%
\begin{lrbox}{\measure@tikzpicture}%
\BODY
\end{lrbox}%
\pgfmathparse{#1/\wd\measure@tikzpicture}%
\BODY
}
\makeatother

\begin{document}
\title{On the Stability of Unconditionally Positive and Linear Invariants Preserving Time Integration Schemes}
\author[1]{Thomas Izgin} 
\author[1]{Stefan Kopecz}
\author[1]{Andreas Meister}
\affil[1]{Department of Mathematics, University of Kassel, Germany}
\affil[]{izgin@mathematik.uni-kassel.de\ \&\ kopecz@mathematik.uni-kassel.de\ \&\ meister@mathematik.uni-kassel.de}
\setcounter{Maxaffil}{0}
\renewcommand\Affilfont{\itshape\small}
\date{}
%

%
%
\maketitle
\begin{abstract} Higher-order time integration methods that unconditionally preserve the positivity and linear invariants of the underlying differential equation system cannot belong to the class of general linear methods. This poses a major challenge for the stability analysis of such methods since the new iterate depends nonlinearly on the current iterate.
Moreover, for linear systems, the existence of linear invariants is always associated with zero eigenvalues, so that steady states of the continuous problem become non-hyperbolic fixed points of the numerical time integration scheme. Altogether, the stability analysis of such methods requires the investigation of non-hyperbolic fixed points for general nonlinear iterations. 

Based on the center manifold theory for maps we present a theorem for the analysis of  the stability of non-hyperbolic fixed points of time integration schemes applied to problems whose steady states form a subspace.
This theorem provides sufficient conditions for both the stability of the method and the local convergence of the iterates to the steady state of the underlying initial value problem.
This theorem is then used to prove the unconditional stability of the MPRK22($\alpha$)-family of modified Patankar-Runge-Kutta schemes when applied to arbitrary positive and conservative linear systems of differential equations. The theoretical results are confirmed by numerical experiments. \end{abstract}
\section{Introduction}\label{sec:intro}
%
In recent years, high order time integration methods have been introduced which, when applied to specific differential equations, unconditionally preserve both positivity and linear invariants, see \cite{BDM2003,KM18,KM18Order3,MR3934688,MR3969000,MR4064785,MR4194400,BBKS2007,BRBM2008,MR4109346,MR4087156}.  
For general linear methods, see \cite{MR2604724,MR2657217}, unconditional positivity is restricted to first order \cite{MR509974,MR1963511} and hence, the previously mentioned schemes do not belong to the class of general linear methods. This has clear implications for the stability analysis of such methods as described below.

In the following we investigate the stability behavior of the above mentioned numerical methods applied to stable positive linear systems of the form
\begin{equation}
\by'(t)=\bA\by(t) \label{eq:Dahlquist_System}
\end{equation}
with $\bA\in \R^{N\times N}$ 
possessing $k>0$ linear invariants and initial condition
\begin{equation}
\by(0)=\by^0>\bzero \label{eq:IC}.
\end{equation}
The presence of linear invariants means, that there exist vectors $\bn_1,\dotsc,\bn_k\in \R^N\setminus\{\bzero\}$ such that $\bn_i^T\by(t)=\bn_i^T\by^0$ for all $t\geq 0$, or equivalently $\bn_i^T\bA=\bzero$ for $i=1,\dotsc,k$. 
Note that the existence of $k$ linear invariants is given if and only if $k=\dim(\ker(\bA^T))=\dim(\ker(\bA))$. 
In addition, the system \eqref{eq:Dahlquist_System} is positive if and only if the matrix $\bA$ is a Metzler matrix, i.\,e.\ a matrix with nonnegative off-diagonal elements, see \cite{Luen79}, which guarantees $\by(t)>\bzero$ for all $t>0$ whenever $\by^0>\bzero$. 
Moreover, to ensure stable steady states $\by^*\in \ker(\bA)$
the matrix $\bA$ in \eqref{eq:Dahlquist_System} must have a spectrum $\sigma(\bA)\tm \C^-=\{z\in \C\mid \operatorname{Re}(z)\leq0\}$ and eigenvalues of $\bA$ with vanishing real part have to be associated with a Jordan block size of 1, see Theorem \ref{Thm:StabLin}.
In this situation we have $\by(t)\to\by^*$ as $t\to\infty$. 
Up to now a stability analysis of higher-order positivity preserving time integration schemes applied to \eqref{eq:Dahlquist_System} is only available for $N=2$ as given in \cite{IKM2122}.

As a consequence of the presence of linear invariants, $0$ is always an eigenvalue of $\bA$ which implies the existence of nontrivial steady state solutions. For every reasonable  time integration scheme $\by^{n+1}=\bg(\by^n)$, these steady state solutions have to be fixed points. The common way to study the stability of a fixed point $\by^*$ of $\bg$ is to compute the eigenvalues of the Jacobian $\bD\bg(\by^*)$. It is well-known that the fixed point $\by^*$ is asymptotically stable if the spectral radius $\rho$ of the Jacobian satisfies $\rho(\bD\bg(\by^*))<1$. Unfortunately, the existence of linear invariants leads to non-hyperbolic fixed points $\by^*$ of the numerical scheme, i.\,e.\ the Jacobian $\bD\bg(\by^*)$ has at least one eigenvalue $\lambda$ with $\abs\lambda = 1$. 

If the time integration scheme applied to \eqref{eq:Dahlquist_System} results in a linear iteration
\[\by^{n+1}=\bm R(\Delta t,\bA)\by^n,\]
as is the case for Runge--Kutta schemes, the stability of the non-hyperbolic fixed point $\by^*$ is again fully determined by the eigenvalues of the Jacobian $\bD\bg(\by^*)=\bm R(\Delta t,\bA)$. 
In this case, the non-hyperbolic fixed point $\by^*$ is stable if and only if $\rho(\bD\bg(\by^*))=1$ and all eigenvalues $\lambda$ of $\bD\bg(\by^*)$ with $\abs{\lambda}=1$ are associated with a Jordan block of size 1, see \cite{MR1912409}.
Unfortunately, the application of higher-order positivity preserving schemes to the linear system \eqref{eq:Dahlquist_System}
results in a nonlinear iteration of the form 
\begin{equation*}
\by^{n+1}=\bm R(\Delta t, \bA,\by^n)\by^n,
\end{equation*}
see \cite{OH17}. 
For such iterations the stability is not fully determined by the eigenvalues of the Jacobian, see for instance \cite{Osipenko2009}.
Hence, the stability analysis of these numerical methods requires the investigation of non-hyperbolic fixed points of a nonlinear iteration. This is significantly more demanding compared to the linear case.

One way to study the stability of non-hyperbolic fixed points of nonlinear iterations is the center manifold theory, see \cite{mccracken1976hopf,carr1982,iooss1979}. This theory states that the stability of a non-hyperbolic fixed point can be determined by studying the iteration on a lower-dimensional invariant manifold, the so-called center manifold.

To avoid the application of the center manifold theory to each positivity preserving scheme separately, we present a theorem which provides sufficient conditions for the stability of all such methods. Thereby, the main assumption of this new theorem is 
that the fixed points of the nonlinear iteration form a linear subspace of $\R^N$. This is a reasonable requirement due to the fact that the steady states of the underlying differential equation \eqref{eq:Dahlquist_System} also form a linear subspace of dimension $k>0$, whenever $k$ linear invariants are present.
The theorem contains two main statements. First, the existence of $k$ linear invariants implies that $\lambda=1$ is an eigenvalue of the Jacobian $\bD\bg(\by^*)$ of multiplicity $k$ and the non-hyperbolic fixed point $\by^*$ is stable, if the remaining $N-k$ eigenvalues have absolute value less than one. Second, if the numerical scheme preservers all $k$ linear invariants, then the iterates locally converge to the unique steady state of the initial value problem \eqref{eq:Dahlquist_System}, \eqref{eq:IC}.
Furthermore, it is worth mentioning that the new theorem can directly be used for the stability analysis of time integration schemes in the context of nonlinear systems of differential equations.  

In addition, we want to emphasize that it is not sufficient to assess the stability of a higher-order positivity preserving scheme in terms of a linear system of the form  
\begin{equation}\label{eq:2x2Dahlquist}
\by'=\Vec{
\lambda & 0\\-\lambda&0} \by, \quad \by(0)=\by^0>\bzero,\quad \lambda\in \R^-,
\end{equation}
which can be seen as a adaptation of Dahlquist's equation
\begin{equation*}
y'=\lambda y,\quad \lambda\in \C^-,
\end{equation*}
originally introduced in \cite{D63}, to linear conservative systems.
One example for this fact is given in \cite{IKM2122}, where the so-called MPRK22ncs($\alpha$) schemes are proven to be $L_0$-stable in the following sense. Applied to the conservative system \eqref{eq:2x2Dahlquist}
the state variable $y_1^n$ satisfies $y_1^{n+1}=R(\Delta t \lambda)y_1^n$ with 
\begin{equation*}
R(z)=\frac{(1-\alpha z)^{1-\frac{1}{\alpha}}}{(1-\alpha z)^{1-\frac{1}{\alpha}}-z\bigl(1-(\alpha-\frac12)z\bigr)},
\end{equation*}
so that $\lim_{z\to -\infty}R(z)=0$ and $\lvert R(z)\rvert\leq1$ for all $z\leq0$ and $\alpha\geq \frac12$. In total this means that the first component represents the behavior of the numerical scheme applied to the Dahlquist equation for $\lambda\in \R^-$ and satisfies all conditions for a scheme to be $L_0$-stable, see \cite{L09}.
Nevertheless, in \cite{IKM2122} it is proved that MPRK22ncs($\alpha$) face severe time step restrictions for $\alpha<1$ in order to be stable when applied to a general two--dimensional linear positive and conservative system
\begin{equation}
\mbfy'=\bA\mbfy,\quad \bA=\begin{pmatrix*}[r]
-a &b\\ a &-b
\end{pmatrix*},\quad a,b\geq 0,\quad a+b> 0,\label{OLD_PDS_test}
\end{equation}
which was also used in \cite{IKM21} for studying the linearization of MPRK22 schemes.
Hence, to understand the stability behavior of such nonlinear schemes, one should directly investigate the system \eqref{eq:Dahlquist_System}. 

Besides the introduction of a stability theorem for general higher-order positivity preserving schemes, 
the usability of this theorem will be demonstrated in the context of MPRK22($\alpha$) methods to theoretically study their known high robustness.
These schemes were introduced for the time integration of positive and conservative production-destruction systems (PDS)
\begin{equation}
y_i'=\sum_{j=1}^N (p_{ij}(\by)-d_{ij}{(\by)})\quad \text{ with } \quad p_{ij}(\by), d_{ij}(\by) \ge 0 \label{eq:PDS_ODE}
\end{equation}
with $\by>\bzero$ and $i=1,\dotsc,N$ in \cite{KM18} and are given by  
\begin{subequations}\label{eq:MPRK22b}
\begin{align}
\quad &\begin{aligned}
\mathllap{y_i^{(1)}} &= y_i^n,
\end{aligned}\label{eq:y_1=y_n}\\ 
&\begin{aligned}
\mathllap{y_i^{(2)}} &= y_i^n + \alpha\Delta t\sum_{j=1}^N\vastl(p_{ij}(\mbfy^{(1)})\frac{y_j^{(2)}}{y_j^{(1)}}-d_{ij}(\mbfy^{(1)})\frac{y_i^{(2)}}{y_i^{(1)}}\vastr),
\end{aligned}\label{eq:gamma}\\ 
&\begin{multlined}[b][.7\columnwidth]
\mathllap{y_i^{n+1}} = y_i^n + \Delta t\sum_{j=1}^N\vastl( \Biggl(\biggl(1-\frac1{2\alpha}\biggr) p_{ij}(\mbfy^{(1)})+\frac1{2\alpha} p_{ij}(\mbfy^{(2)})\Biggr)\frac{y_j^{n+1}}{(y_j^{(2)})^{\frac{1}{\alpha}}(y_j^{(1)})^{1-\frac{1}{\alpha}}}\\
- \Biggl(\biggl(1-\frac1{2\alpha}\biggr) d_{ij}(\mbfy^{(1)})+ \frac1{2\alpha} d_{ij}(\mbfy^{(2)})\Biggr)\frac{y_i^{n+1}}{(y_i^{(2)})^{\frac{1}{\alpha}}(y_i^{(1)})^{1-\frac{1}{\alpha}}}\vastr)
\end{multlined}\label{eq:y_n_plus_1}
\end{align}
\end{subequations}
for $i=1,\dots,N$ with $\alpha\geq\frac12$. These schemes can be applied to \eqref{eq:Dahlquist_System} whenever the linear system is conservative, which means $\mathbf 1^T\by$ with $\mathbf 1=(1,\dots,1)^T$ represents a linear invariant. 
It was shown in \cite{KM18} that MPRK22($\alpha$) schemes are unconditionally positive and conservative, hence the iterates satisfy $\by^n>\bzero$ for all $n\in \N$ whenever $\by^0>\bzero$ as well as $\mathbf{1}^T\by^n=\mathbf{1}^T\by^0$. 
Generally, we say that a scheme unconditionally conserves the linear invariant determined by a vector $\bn\in \R^N$ if $\bn^T\by^n=\bn^T\by^0$ for all $n\in \N$ and $\Delta t>0$.

The second order MPRK22($\alpha$) schemes are examples of modified Patankar--Runge--Kutta methods, see \cite{BDM2003,KM18,KM18Order3,MR3969000,MR3934688,MR4064785}. In \cite{KM18,KM18Order3,MR3987250}, MPRK schemes up to third order were introduced and investigated which led to the construction of SSP-MPRK methods in \cite{MR3969000,MR3934688}. Based on deferred correction methods and the idea in \cite{BDM2003}, arbitrary high order MPRK schemes where introduced in \cite{MR4064785}.
MPRK schemes are of considerable interest and widely used such as in the context of ecosystems \cite{HenseBeckmann2010,HenseBurchard2010,WHK2013,BMZ2007,BMZ2009,MeisterBenz2010} or ocean models \cite{SemeniukDastoor2017,BBKMNU2006}. Further applications can be found in the context of magneto-thermal winds \cite{Gressel2017} or warm-hot intergalactic mediums \cite{KlarMuecket2010} as well as in that of the SIR epidemic model \cite{WS21}. 
For other recent approaches which facilitate positive and conservative numerical approximations, we refer to \cite{MR4087156,MR4109346, MR4194400,nuesslein2021positivitypreserving,blanes2021positivitypreserving}, some of which even conserve all linear invariants.

The outline of the paper is as follows. 
In Section~\ref{sec:stab_dyn_sys} we  
present the novel stability theorem, which gives sufficient conditions for higher-order positivity preserving schemes to ensure stability as well as the local convergence to the unique steady state of the underlying initial value problem.
The application of this theorem to MPRK22($\alpha$) schemes is subject of Section~\ref{sec:stab_mprk22}. We show that MPRK22($\alpha$) methods are unconditionally stable when applied to linear systems. Finally in Section \ref{Sec:Num_Tests}, we provide numerical experiments confirming the theoretical results.

\section{Center Manifold Theory and Stability of Time Integration Schemes}\label{sec:stab_dyn_sys}
For the sake of completeness, we summarize in this section the main definitions and statements concerning the stability of steady states and fixed points of maps.
This theory is used to prove the main theorem of this work, Theorem \ref{Thm_MPRK_stabil}, which provides  criteria to assess the stability of non-hyperbolic fixed points of nonlinear iterations conserving at least one linear invariant.

In the following, we use $\norm{\ \cdot\  }$ to represent an arbitrary norm in $\R^l$ for $l\in \N$ and $\bD\bf$ denotes the Jacobian of a map $\bf$.
\begin{defn}\label{Def Lyapunov Cont}
Let $\by^*\in \R^N$ be a steady state solution of a differential equation $\by'={\bf}(\by)$, that is ${\bf}(\by^*)=\bzero$.
\begin{enumerate}
\item\label{item:lyap_stab} Then $\by^*$ is called \emph{Lyapunov stable} if, for any $\epsilon>0$, there exists a $\delta=\delta(\epsilon)>0$ such that $\norm{\by(0)-\by^*}<\delta$ implies $\norm{\by(t)-\by^*}<\epsilon$ for all $t\geq 0$.
\item If in addition to a), there exists a constant $c>0$ such that $\Vert \by(0)-\by^*\Vert<c$ implies  $\Vert \by(t)-\by^*\Vert \to 0$ for $t\to \infty$,  we call  $\by^*$ \emph{asymptotically stable.}
\item A steady state solution that is not stable is said to be \emph{unstable}.
\end{enumerate}
\end{defn}

\begin{thm}{(\cite[Theorem 3.23]{MR1912409})}\label{Thm:StabLin}
A steady state $\by^*$ of $\by'=\bM\by$ with a matrix $\bM\in \R^{N\times N}$
\begin{enumerate}
\item 	 is stable if and only if $\max_{\lambda\in \sigma(\bM)}\operatorname{Re}(\lambda)\leq 0$ and all $\lambda$ with $\operatorname{Re}(\lambda)=0$ are associated with a Jordan block of size 1.
\item  is asymptotically stable if and only if $\max_{\lambda\in \sigma(\bM)}\operatorname{Re}(\lambda)< 0$.
\end{enumerate}
\end{thm}
Now, according to \eqref{eq:Dahlquist_System}, the eigenvalue $\lambda=0$ of $\bA$ has a multiplicity of $k\geq 1$. Hence, the corresponding differential equation has no asymptotically stable steady state $\by^*$. However, since we have assumed that $\sigma(\bA)\tm \C^-$ and the eigenvalues with zero real part  are associated with a Jordon block of size 1, any steady state $\by^*$ of \eqref{eq:Dahlquist_System} is stable.

As we are interested in numerical schemes mimicking the stability behavior of the exact solution, we shall consider the following definition.
\begin{defn}\label{Def_Lyapunov_Diskr}
Let $\by^*$ be a fixed point of an iteration scheme $\by^{n+1}=\bg(\by^n)$, that is $\by^*=\bg(\by^*)$. 
\begin{enumerate}
\item\label{def:stab} Then $\by^*$ is called \emph{Lyapunov stable} if, for any $\epsilon>0$, there exists a $\delta=\delta(\epsilon)>0$ such that $\norm{\by^0-\by^*}<\delta$ implies $\norm{\by^n- \by^*}<\epsilon$ for all $n\geq 0$.
\item If in addition to a), there exists a constant $c>0$ such that $\Vert \by^0-\by^*\Vert<c$ implies $\Vert \by^n-\by^*\Vert \to 0$ for $n\to \infty$, we call $\by^*$ \emph{asymptotically stable.}
\item A fixed point that is not stable is said to be \emph{unstable}.
\end{enumerate}
\end{defn}
In the following, we will also briefly speak of stability instead of Lyapunov stability. As stated by the next theorem, it is in some cases sufficient to investigate the linearized method in order to understand the stability properties of a fixed point.
\begin{thm}[{\cite[Theorem 1.3.7]{SH98}}]\label{Thm:_Asym_und_Instabil}
Let  $\by^{n+1}=\bg(\by^n)$ be an iteration scheme with fixed point $\by^*$. Suppose the Jacobian $\bD\bg(\by^*)$ exists and denote its spectral radius by $\rho(\bD\bg(\by^*))$. Then
\begin{enumerate}
\item $\by^*$ is asymptotically stable if $\rho(\bD\bg(\by^*))<1$. 
\item $\by^*$ is unstable if $\rho(\bD\bg(\by^*))>1$.
\end{enumerate}
\end{thm}
The above theorem gives sufficient conditions for the stability of fixed points that are hyperbolic in the following sense.
\begin{defn}[{\cite[Definition 1.3.6]{SH98}}]\label{Def hyperbolic}
A fixed point $\by^*$ of an iteration scheme $\by^{n+1}=\bg(\by^n)$ is called \emph{hyperbolic} if $\abs\lambda\ne 1$ for all eigenvalues $\lambda$ of $\bD\bg(\by^*)$. If a fixed point is not hyperbolic, it is called \emph{non-hyperbolic}.
\end{defn}
The stability of non-hyperbolic fixed points of a scheme outside the class of general linear methods is in general not induced by the eigenvalues of the corresponding Jacobian, see \cite{Osipenko2009, SH98}. Hence, higher-order terms have to be included within the stability analysis of this kind of methods. 
One possibility to decrease the complexity of such a stability analysis is to use the center manifold theory, which allows to assess the stability based on a corresponding iteration on a lower dimensional manifold and is briefly summarized in the next subsection.
\subsection{Center Manifold Theory}\label{Sec:Center_Manifold}
To study the stability of a non-hyperbolic fixed point $\by^*$ of an iteration scheme with $\mathcal C^1$-map $\bg$, we make use of an affine linear transformation\footnote{See the proof of Theorem~\ref{Thm_MPRK_stabil} for the details of this transformation.} to obtain a $\mathcal{C}^1$-map $\bG\colon \mathcal M\to \R^{N}$, with $\mathcal M\subset \R^{N}$ being a neighborhood of the origin, which has the form
\begin{equation}
\bG(\bw_1,\bw_2)=\Vec{\bU\bw_1+\bu(\bw_1,\bw_2)\\\bV\bw_2 +\bv(\bw_1,\bw_2)},\label{Form_of_G}
\end{equation}
with $\bw_1\in \R^m$, $\bw_2\in \R^l$ and $m+l=N$. The square matrices $\bU\in\R^{m\times m}$ and $\bV\in\R^{l\times l}$ are such that $\abs\lambda = 1$ holds for all eigenvalues $\lambda$ of $\bU$ and each eigenvalue $\mu$ of $\bV$  satisfies $\abs\mu <1$. The functions $\bu$ and $\bv$ are in $\mathcal{C}^1$ and $\bu,\bv$ as well as their first order derivatives vanish at the origin, that is
\begin{align*}
\bu(\bzero,\bzero)&=\bzero, & \bD\bu(\bzero,\bzero)&=\bzero, & 
\bv(\bzero,\bzero)&=\bzero, & \bD\bv(\bzero,\bzero)&=\bzero,
\end{align*} where $\bzero$ stands for the zero vector or matrix of appropriate size, respectively. 
In particular, the fixed point $\by^*$ of $\bg$ is mapped to $\bzero$, which is a fixed point of $\bG$ with equal stability properties as $\by^*$ as  we point out in the proof of Theorem \ref{Thm_MPRK_stabil}. 
Hence, it is sufficient to study the stability of the origin with respect to $\bG$, which is a simplification due to the existence of a center manifold. 

\begin{thm}{(Center Manifold Theorem, \cite[Theorem 2.1, Remark 2.6]{mccracken1976hopf})}\label{Thm:Ex_CM}
Let $\bG$ be defined as in \eqref{Form_of_G} with Lipschitz continuous derivatives on $\mathcal M$.
\begin{enumerate}
\item\label{item:existence}  (Existence):
There exists a center manifold for $\bG$, which is locally representable as the graph of a function $\bh\colon\R^m\to \R^l$.  This means, for some $\epsilon>0$ there exists a $\mathcal{C}^1$-function $\bh\colon\R^m\to\R^l$ with $\bh(\bzero)=\bzero$ and $\bD\bh(\bzero)=\bzero$ such that $\Vert \bw_1^0\Vert,\Vert \bw_1^1\Vert  <\epsilon$ and $(\bw_1^1, \bw_2^1)^T=\bG(\bw_1^0,\bh(\bw_1^0))$ imply $\bw_2^1=\bh(\bw_1^1)$.
\item\label{item:attractivity}  (Local Attractivity): If in addition to \ref{item:existence} the iterates $(\bw_1^n,\bw_2^n)^T$ generated by 
\begin{align}
\Vec{\bw_1^{n+1}\\\bw_2^{n+1}}=\bG(\bw_1^n,\bw_2^n)=\Vec{\bU\bw_1^n+\bu(\bw_1^n,\bw_2^n)\\\bV\bw_2^n +\bv(\bw_1^n,\bw_2^n)},\quad \Vec{\bw_1^0\\ \bw_2^0}\in \mathcal M\label{Form_of_G_transformed}.
\end{align}  
satisfy $\Vert \bw_1^n\Vert,\Vert \bw_2^n\Vert<\epsilon$ for all $n\in \N_0$, then the distance of $(\bw_1^n,\bw_2^n)$ to the center manifold tends to zero for $n\to \infty$, i.\,e.\ $\Vert \bw_2^n-\bh(\bw_1^n)\Vert\to 0$ for $n\to \infty$.
\end{enumerate}
\end{thm}
We want to note that the above theorem is formulated with weaker assumptions than the corresponding theorem in \cite{IKM2122}. 
This allows to use Theorem~\ref{Thm_MPRK_stabil} introduced below for the stability analysis of a larger class of time integration schemes than is possible with \cite[Theorem 2.9]{IKM2122}.

As will be seen in Theorem \ref{Thm:Stab_CM}, the existence of a center manifold enables the investigation of the stability properties of the origin based on a system with reduced dimension.  This reduced system is obtained by restricting \eqref{Form_of_G} to the center manifold, i.\,e.\  using $\bw_2=\bh(\bw_1)$ which leads to the map
\begin{equation}\label{Flow_center_manifold} \mathcal G(\bw_1)=\bU \bw_1 + \bu(\bw_1,\bh(\bw_1)).
\end{equation}
\begin{thm}{(\cite[Theorem 8]{carr1982})}\label{Thm:Stab_CM} (Stability): Suppose the fixed point $\bzero\in \R^m$ of $\mathcal G$ from \eqref{Flow_center_manifold} is stable, asymptotically stable or unstable. Then the fixed point $\bzero \in \R^N$ of $\bG$ from \eqref{Form_of_G} is stable, asymptotically stable or unstable, respectively.
\end{thm}
In summary,  the stability of a non-hyperbolic fixed point $\by^*\in\R^N$ of a map $\bg$ can be determined by investigating the fixed point $\bzero\in\R^m$ of $\mathcal G$, which has a lower complexity due to the reduced dimension $m<N$.

To actually calculate the center manifold we need to solve
\begin{equation*}
(\bw_1^1, \bh(\bw_1^1))^T=\bG(\bw_1^0,\bh(\bw_1^0))=\Vec{\bU\bw_1^0+\bu(\bw_1^0,\bh(\bw_1^0))\\\bV\bh(\bw_1^0) +\bv(\bw_1^0,\bh(\bw_1^0))},
\end{equation*} 
which can be rewritten as
\[\bh(\bU\bw_1^0+\bu(\bw_1^0,\bh(\bw_1^0)))=\bV\bh(\bw_1^0) +\bv(\bw_1^0,\bh(\bw_1^0))\]
The above invariance property offers a way to approximate the center manifold up to an arbitrary order.
\begin{thm}{(\cite[Theorem 7]{carr1982})}\label{Thm:Comp_func_h}
Let $\bh$ be a center manifold for $\bG$ and $\bm\Phi$ be a $\mathcal{C}^1(\R^m, \R^l)$-map with $\bm\Phi(\bzero)=\bzero$ and $\bD\bm\Phi(\bzero)=\bzero$. If
\[\bm\Phi(\bU\bw_1+\bu(\bw_1,\bm\Phi(\bw_1)))-\left(\bV\bm\Phi(\bw_1) +\bv(\bw_1,\bm\Phi(\bw_1))\right)=\mathcal{O}(\Vert \bw_1 \Vert^q)\]  as $\bw_1\to \bzero$ for some $q>1$, then $\bh(\bw_1)=\bm\Phi(\bw_1)+\mathcal{O}(\Vert \bw_1 \Vert^q)$ as $\bw_1\to \bzero$.
\end{thm}
\subsection{Stability of numerical schemes with linear invariants}
In this subsection we make use of the center manifold theory to investigate the stability of fixed points $\by^*$ of a numerical scheme $\by^{n+1}=\bg(\by^n)$ with $\bg\from D\to D$ and $D\tm \R^N$. Thereby, we assume that there exists a neighborhood $\mathcal D\tm D$ of $\by^*$ such that $\bg\big|_\mathcal{D}\in \mathcal C^1$ has first derivatives that are Lipschitz continuous on $\mathcal{D}$, so we can apply Theorem \ref{Thm:Ex_CM}. Based on this assumption, the following Theorem~\ref{Thm_MPRK_stabil} yields a sufficient condition for the Lyapunov stability of $\by^*$ based on the eigenvalues of the corresponding Jacobian $\bD\bg(\by^*)$. If $\bg$ in addition conserves all linear invariants of $\bA$ from \eqref{eq:Dahlquist_System}, i.\,e.\ $\bn^T\bg(\by)=\bn^T\by$ for all $\by\in D$ whenever $\bn^T\bA=\bzero$, then Theorem \ref{Thm_MPRK_stabil} also states that the numerical scheme locally convergences towards the unique steady state $\by^*$ of \eqref{eq:Dahlquist_System}, \eqref{eq:IC}.

Compared to \cite[Theorem 2.9]{IKM2122}, it is worth mentioning that due to Theorem \ref{Thm_MPRK_stabil} we  are no longer restricted to focus on positive as well as conservative schemes, since the new statement allows for considering a general map $\bg:D\to D$. 
An essential advantage is given by weakening the assumption $\bg\in \mathcal C^2$ to the requirement that $\bg$ has to represent a $\mathcal C^1$-map with Lipschitz continuous first derivatives, which opens up a much wider application, where now in particular even GeCo methods \cite{MR4087156} as well as BBKS schemes \cite{BBKS2007,BRBM2008,MR4109346} can be investigated. 

For a compact notation we introduce the matrix
\begin{equation}\label{eq:N}
\bN=\begin{pmatrix}
\bn_1^T\\
\vdots\\
\bn_k^T
\end{pmatrix}\in \R^{k\times N} 
\end{equation}
with $\bn_1,\dotsc,\bn_k$ being a basis of $\ker(\bA^T)$ as well as the set
\begin{equation}
H=\{\by\in \R^N\mid \bN\by=\bN\by^*\}\label{eq:H}
\end{equation}
and point out that $\by\in H\cap D$ implies $\bg(\by)\in H\cap D$, if and only if $\bg$ conserves all linear invariants.

\begin{thm}\label{Thm_MPRK_stabil}
Let $\bA\in \R^{N\times N}$ such that $\ker(\bA)=\Span(\bv_1,\dotsc,\bv_k)$ represents a $k$-dimensional subspace of $\R^N$ with $k>0$. Also, let $\by^*\in \ker(\bA)$ be a fixed point of $\bg\from D\to D$ where $D\tm \R^N$ contains a neighborhood $\mathcal D$ of $\by^*$. Moreover, let any element of $C=\ker(\bA)\cap \mathcal D$ be a fixed point of $\bg$ and suppose that $\bg\big|_\mathcal{D}\in \mathcal C^1$ as well as that the first derivatives of $\bg$ are Lipschitz continuous on $\mathcal{D}$. Then $\bD\bg(\by^*)\bv_i=\bv_i$ for $i=1,\dotsc, k$ and the following statements hold.
\begin{enumerate}
\item\label{it:Thma} If the remaining $N-k$ eigenvalues of $\bD\bg(\by^*)$ have absolute values smaller than $1$, then $\by^*$ is stable.\label{It:Thm_Stab_a}
\item\label{it:Thmb} Let $H$ be defined by \eqref{eq:H} and $\bg$ conserve all linear invariants, which means that $\bg(\by)\in H\cap D$ for all $\by\in H\cap D$. If additionally the assumption of \ref{It:Thm_Stab_a} is satisfied, then there exists a $\delta>0$ such that $\by^0\in H\cap D$ and $\norm{\by^0-\by^*}<\delta$ imply $\by^n\to \by^*$ as $n\to \infty$.
\end{enumerate}
\end{thm}
Before we prove the above theorem we want to emphasize in the next remark that its application is not restricted to linear systems of differential equations \eqref{eq:Dahlquist_System}.
\begin{rem}\label{Rem:NonlinearAppl}
Let us consider a general system of ordinary differential equations $\by'=\bf(\by)\in \R^N$ with $k>0$ linear invariants determined by $\bn_1,\dotsc,\bn_k$ and a $k$--dimensional subspace $\mathcal V=\Span(\bv_1,\dotsc,\bv_k)\tm\{\by\in \R^N\mid \bf(\by)=\bzero\}$. Consequently, we can construct a matrix $\bA$ such that $\ker(\bA)=\mathcal V$ and $\ker(\bA^T)=\Span(\bn_1,\dotsc,\bn_k)$ as follows, and thus apply Theorem \ref{Thm_MPRK_stabil}.

As $\bA$ is uniquely determined by its operation on a basis of $\R^N$ we first set $\bA\bv_i=\bzero$  for $i=1,\dotsc,k$  so that $\ker(\bA)=\mathcal V$ is satisfied. To find an expression for $\im(\bA)$ we make use  of  $\im(\bA)=(\ker(\bA^T))^\perp=(\Span(\bn_1,\dotsc,\bn_k))^\perp$. Using the matrix notation \eqref{eq:N}, this means that $\bs\in \im(\bA)$ if and only if $\bN\bs=\bzero$, or equivalently $\bs\in \ker(\bN)$. Since $\dim(\ker(\bN))=N-k$, there exist linearly independent vectors $\bs_1,\dotsc,\bs_{N-k}$ with $\im(\bA)=\Span(\bs_1,\dotsc,\bs_{N-k})$, and hence, there exist linearly independent vectors $\bw_1,\dotsc,\bw_{N-k}$ such that $\bA\bw_i=\bs_i$ for $i=1,\dotsc, N-k$. As a consequence, setting $\mathcal{W}=\Span(\bw_1,\dotsc,\bw_{N-k})\tm \R^N$ yields $\mathcal{V}\oplus\mathcal{W}=\R^N$.
Altogether, $\mathcal V$ and $\mathcal W$ uniquely determine  the matrix $\bA$ satisfying $\ker(\bA)=\mathcal V$ and $\ker(\bA^T)=\Span(\bn_1,\dotsc,\bn_k)$. Hence, Theorem \ref{Thm_MPRK_stabil} is not restricted to linear systems.
\end{rem}

\begin{proof}[Proof of Theorem \ref{Thm_MPRK_stabil}]
First, we show $\bD\bg(\by^*)\bv_i=\bv_i$ for $i=1,\dotsc,k$. Since $\bg$ is differentiable in $\by^*\in \mathcal D$ the directional derivatives $\partial_\bv \bg(\by^*)=\bD\bg(\by^*)\bv$ exist for all directions $\bv\in\R^N$ and for $i=1,\dotsc,k$ we find
\[\bD\bg(\by^*)\bv_i=\partial_{\bv_i}\bg(\by^*)=\lim_{h\to 0}\frac1h\bigl(\bg(\by^*+h\bv_i)-\bg(\by^*)\bigr). \]
For $\lvert h\rvert$ small enough, we see that $\by^*+h\bv_i\in C$ because of the following. First of all $\by^*+h\bv_i\in\ker(\bA)$ holds for all $h\in\R$, so that we have to show that $\by^*+h\bv_i\in\mathcal D$ for $\lvert h\rvert$ small enough. Since $\by^*\in \mathcal{D}$, there exists a $\gamma>0$ such that the open ball $B_\gamma(\by^*)$ with center $\by^*$ and radius $\gamma$ satisfies $B_\gamma(\by^*)\tm \mathcal D$. Choosing $\lvert h\rvert<\frac{\gamma}{\Vert \bv_i\Vert}$ we find \[ \Vert \by^*+h\bv_i-\by^*\Vert \leq \lvert h\rvert \Vert \bv_i\Vert <\gamma,\] such that $\by^*+h\bv_i\in\ker(\bA)\cap B_\gamma(\by^*)\tm \ker(\bA)\cap\mathcal D=C$ is a fixed point of $\bg$. Hence,
\[\bD\bg(\by^*)\bv_i=\lim_{h\to 0}\frac1h\bigl(\by^*+h\bv_i-\by^*\bigr)=\bv_i, \]
which shows that $\bv_i$ is an eigenvector of $\bD\bg(\by^*)$ with associated eigenvalue $1$.
Thus, the spectrum of $\bD\bg(\by^*)$ contains the eigenvalue $1$ with a multiplicity of at least $k$.

\begin{enumerate} 
\item\label{item:stability}
We now assume that the remaining $N-k$ eigenvalues  of $\bD\bg(\by^*)$ have absolute values smaller than 1.
Next we introduce the matrix of generalized eigenvectors $\bS$ where the first $k$ columns are given by the basis vectors $\bv_1,\dotsc, \bv_k$ of $\ker(\bA)$. Thus, we obtain
\begin{equation}\label{eq:DG_digaonal}
\bS^{-1}\bD\bg(\by^*)\bS=\bJ
\end{equation}
with the Jordan normal form $\bJ$ of $\bD\bg(\by^*)$. We want to point out that the upper left $k\times k$ block of $\bJ$ is the identity matrix, since the $k$ basis vectors $\bv_1,\dotsc, \bv_k$ of $\ker(\bA)$ are eigenvectors with associated eigenvalue $1$.

We want to use the Theorem~\ref{Thm:Ex_CM}~\ref{item:existence} in combination with Theorem~\ref{Thm:Stab_CM} to conclude that $\by^*$ is a stable fixed point.
The theorems require a map $\bG$ of form \eqref{Form_of_G}, which shall be obtained from $\bg$ by means of an affine linear transformation.
We consider the affine transformation \[\bT\from\R^N\to\R^N,\quad \by\mapsto\bw=\bT(\by)=\bS^{-1}(\by-\by^*),\] where
the inverse transformation $\bT^{-1}$ is given by $\bT^{-1}(\bw)=\bS\bw+\by^*$.
By construction, $\ker(\bA)$ is mapped onto the subspace spanned by the first $k$ unit vectors $\be_1,\dotsc,\be_k$ of $\R^N$, as for $\by^*=\sum_{i=1}^kt_i\bv_i\in \ker(\bA)$ we find
\begin{equation*}
\begin{aligned}
\bT\left(\sum_{i=1}^kr_i\bv_i\right)&=\bS^{-1}\left(\sum_{i=1}^kr_i\bv_i-\by^*\right)=\bS^{-1}\left(\sum_{i=1}^k(r_i-t_i)\bv_i\right)\\
&=\sum_{i=1}^k(r_i-t_i)\bS^{-1}\bv_i=\sum_{i=1}^k(r_i-t_i)\be_i
\end{aligned}
\end{equation*}
for arbitrary choices of $r_1,\dotsc,r_k\in \R$.
In particular, $\by^*$ is mapped to the origin.

In order to use Theorem \ref{Thm:Ex_CM}, we have  to define an appropriate $\mathcal C^1$-map $\bG\from \mathcal M\to \R^N$. Therefore we define $\mathcal M=\bT(\mathcal D)$ which is a neighborhood of the origin since $\bT$ is an invertible affine linear map. In particular, we use
\begin{equation}
\bG\from \bT(\mathcal D)\to \R^N,\quad \bG(\bw) = \bT(\bg(\bT^{-1}(\bw)))\label{eq:TgTinv}
\end{equation}
and observe that the origin is a fixed point of $\bG$. To represent $\bG$ in the form \eqref{Form_of_G}, we use $\bg(\by^*)=\by^*$ and write $\bg$ as
\begin{equation}
\begin{aligned}
\bg(\by)&=\bg(\by^*)+\bD\bg(\by^*)(\by-\by^*)+\bR(\by)\\&=\by^*+\bD\bg(\by^*)(\by-\by^*)+\bR(\by),\label{yn+1_Lagrange_remainder}
\end{aligned}
\end{equation}
where the remainder $\bR(\by)$ can be written as
\begin{align}
\bR(\by)=\bg(\by)-\by^*-\bD\bg(\by^*)(\by-\by^*).\label{Lagrange_remainder}
\end{align}
In particular, we have
\begin{equation}
{\bR}(\by^*)=\bzero,\quad
\bD {\bR}(\by^*)=\bzero.\label{R(0)=0,DR(0)=0}
\end{equation}
By inserting \eqref{yn+1_Lagrange_remainder} in \eqref{eq:TgTinv} we obtain
\begin{equation*}
\begin{aligned}
\bG(\bw) &= \bS^{-1}\bigl(\bD\bg(\by^*)(\bT^{-1}(\bw)-\by^*)+\bR(\bT^{-1}(\bw))\bigr)\\
&=\bS^{-1}\bD\bg(\by^*)\bS\bw + \bS^{-1}\bR(\bT^{-1}(\bw))
\end{aligned}
\end{equation*}
and using \eqref{eq:DG_digaonal} yields
\begin{equation}\label{eq:G}
\bG(\bw) = \bJ\bw + \bS^{-1}\bR(\bT^{-1}(\bw))=\begin{pmatrix}
\bI &\\
& \bm R 
\end{pmatrix}\bw + \bS^{-1}\bR(\bT^{-1}(\bw)),
\end{equation}
where $\bI\in \R^{k\times k}$ and $\bm R\in \R^{(N-k) \times (N-k)}$ and $\rho(\bm R)<1$ as $N-k$ eigenvalues of $\bD\bg(\by^*)$ have absolute values smaller than $1$.
Setting $\bw=(\bw_1,\bw_2)^T$ with $\bw_1\in \R^{k}$, $\bw_2\in \R^{N-k}$ and $(\bw_1,\bw_2)\in \bT(\mathcal D)$, \eqref{eq:G} can be rewritten as
\begin{equation}\label{eq:G_form}
\bG(\bw_1,\bw_2)=\Vec{\bU \bw_1+\bu(\bw_1,\bw_2)\\\bV \bw_2+\bv(\bw_1,\bw_2)}
\end{equation} 
with
\begin{equation}\label{eq:UVuv}
\begin{aligned}
\bU&=\bI,  & \bu(\bw_1,\bw_2)&=\bigl(\bS^{-1}{\bR}(\bT^{-1}(\bw_1,\bw_2))\bigr)_{1:k},\\
\bV&= \bm R, & \bv(\bw_1,\bw_2)&=\bigl(\bS^{-1}{\bR}(\bT^{-1}(\bw_1,\bw_2))\bigr)_{k+1:N,}
\end{aligned}
\end{equation}
where we defined $\bv_{l:m}=(v_l,\dotsc,v_m)^T$ for a vector $\bv$ and $l\leq m$.
Each eigenvalue of $\bU$ has absolute value 1 and those of $\bV$ have absolute values smaller than $1$. Furthermore, utilizing $\bT^{-1}(\bzero,\bzero)=\by^*$ we conclude from \eqref{R(0)=0,DR(0)=0} that
$\bu(\bzero,\bzero) = \bv(\bzero,\bzero)=\bzero$, since $\bR(\by^*)=\bzero$, and $\bD \bu(\bzero,\bzero) = \bD \bv(\bzero,\bzero)=\bzero$, since $\bD\bR(\by^*)=\bzero$.
Altogether this demonstrates that \eqref{eq:G} is of form \eqref{Form_of_G}, which is necessary for applying the center manifold theory.

Now, the center manifold theorem~\ref{Thm:Ex_CM}~\ref{item:existence} states that for some $\epsilon>0$ there exists a $\mathcal{C}^1$ function $\bh\from\R^k\to~ \R^{N-k}$ with $\bh(\bzero)=\bzero$ and $\bD\bh(\bzero)=\bzero$, such that $(\bw_1^1,\bw_2^1)^T=\bG(\bw_1^0,\bh(\bw_1^0))$  implies $\bw_2^1=\bh(\bw_1^1)$ for $\norm{\bw_1^0} ,\norm{\bw_1^1}<\epsilon$. 

In the following we make use of the fact that the center manifold is given by 
\begin{equation}
\{(\bw_1,\bw_2)\in\R^N\mid \bw_2=\bzero,\ \Vert \bw_1\Vert <\epsilon\},\label{eq:CM}
\end{equation}
i.\,e.\ $\bh(\bw_1)=\bzero$, for a sufficiently small $\epsilon>0$, which can be shown with Theorem~\ref{Thm:Comp_func_h}.
The function $\bm\Phi\from\R^k\to\R^{N-k}$, $\bm\Phi(\bw_1)=\bzero$ satisfies $\bm\Phi(\bzero)=\bzero$ and $\bD\bm\Phi(\bzero)=\bzero$. In order to compute $\bh$ we first prove that all points $(\bw_1,\bzero)\in \bT(\mathcal D)$ are fixed points of $\bG$. Note, that points $(\bw_1,\bzero)\in \bT(\mathcal D)$ even satisfy
\begin{equation*}
\begin{aligned}
\bT^{-1}(\bw_1,\bzero)&=\bT^{-1}(\sum_{i=1}^k(\bw_1)_i\be_i)=\sum_{i=1}^k(\bw_1)_i\bS\be_i+\by^*=\sum_{i=1}^k(\bw_1)_i\bv_i+\by^*\in \mathcal D\cap \ker(\bA)=C.
\end{aligned}
\end{equation*}
Hence, we find 
\begin{equation}\label{eq:G(w1,0)=(w1,0)}
\bG(\bw_1,\bzero)=  \bT\left(\bg\left(\bT^{-1}\left(\bw_1,\bzero\right)\right)\right)
=\bT\left(\bT^{-1}\left(\bw_1,\bzero\right)\right)=(\bw_1,\bzero)^T.
\end{equation}
Thus, it follows that
\begin{equation*}
\begin{aligned}
\bm\Phi(\bU \bw_1+\bu(\bw_1,\bm\Phi(\bw_1)) &- \left(\bV \bm\Phi(\bw_1)+\bv(\bw_1,\bm\Phi(\bw_1))\right)\overset{\eqref{eq:G_form}}{=}-(\bG(\bw_1,\bzero))_{k+1:N}=\bzero.
\end{aligned}
\end{equation*}
By Theorem~\ref{Thm:Comp_func_h}, $\bm\Phi$ is an approximation of $\bh$ for any order $q>1$. Thus, 
\begin{equation*}
\bh(\bw_1)=\bm\Phi(\bw_1)=\bzero \text{ for }\norm{\bw_1}<\epsilon.
\end{equation*}
To investigate the stability of $\by^*$, we can now consider the map
\begin{equation*}
\mathcal  G(\bw_1)=\bU \bw_1 + \bu(\bw_1,\bh(\bw_1))=\bU \bw_1 + \bu(\bw_1,\bzero)
\end{equation*}
for $\norm{\bw_1}<\epsilon$, where $\bU$ and $\bu$ are given in \eqref{eq:UVuv}.
According to Theorem~\ref{Thm:Stab_CM}, the fixed point $\bzero\in\R^N$ of $\bG$ is stable, if the fixed point $\bzero\in\R^k$ is a stable fixed point of $\mathcal G$.
From \eqref{eq:G(w1,0)=(w1,0)} we see
\begin{equation*}
\mathcal G(\bw_1)=\left(\bG(\bw_1,\bzero)\right)_{1:k}=\bw_1,
\end{equation*}
which implies $\bw_1^{n}=\mathcal G(\bw_1^{n-1})=\bw_1^0$ for all $n\in\N$ and every $\bw_1^0$ with $\norm{\bw_1^0}<\epsilon$. Consequently, for every $\widetilde\epsilon>0$ we define $\widetilde\delta = \min(\widetilde\epsilon,\epsilon)$ to obtain that $\norm{\bw_1^0}<\widetilde\delta$ implies $\norm{\bw_1^n}=\norm{\bw_1^0}<\widetilde\delta\leq \widetilde\epsilon$. Thus, $\bzero\in \R^k$ is a stable fixed point of $\mathcal G$ in the sense of Definition~\ref{Def_Lyapunov_Diskr}~\ref{def:stab}.
Furthermore, by Theorem~\ref{Thm:Stab_CM} the fixed point $\bzero\in\R^N$ of $\bG$ is stable as well.

As a last step, we show that the above conclusions imply that $\by^*$ is a stable fixed point of $\bg$.
We know that $\bzero$ is a stable fixed point of the iteration scheme $\bw^{n+1}=\bG(\bw^n)$, that is for every $\epsilon_w>0$ exists $\delta_w>0$ such that $\norm{\bw^0}<\delta_w$ implies $\norm{\bw^n}<\epsilon_w$. Now, let $\epsilon>0$ be arbitrary, we define $\epsilon_w=\epsilon/\norm{\bS}$ and  $\delta=\delta_w/\norm{\bS^{-1}}$. Hence, if $\norm{\by^0-\by^*}<\delta$, then \[\norm{\bw^0}=\norm{\bT(\by^0)}=\norm{\bS^{-1}(\by^0-\by^*)}\leq \norm{\bS^{-1}}\norm{\by^0-\by^*}<\norm{\bS^{-1}}\delta=\delta_w\] and consequently $\norm{\bw^n}<\epsilon_w$. Furthermore, $\bw^n=\bT(\by^n)=\bS^{-1}(\by^n-\by^*)$ is equivalent to $\bS\bw^n=\by^n-\by^*$ and hence,  $\norm{\by^n-\by^*}\leq\norm{\bS}\norm{\bw^n}<\norm{\bS}\epsilon_w=\epsilon$. Thus, we have shown that $\by^*$ is a stable fixed point of the iteration scheme $\by^{n+1}=\bg(\by^n)$.
\item
Recall from \eqref{eq:H} that $H=\{ \by\in \R^N\mid \bN\by=\bN\by^*\}$ and let $\by^0\in H\cap D$, where $\bN$ is given by \eqref{eq:N}. Note, that $\dim(H)=N-k$ as $\bN$ has rank $k$, and $\by^n\in H$ for all $n\in \N_0$ since $\bg(\by)\in H$ for all $\by\in H\cap D$. Moreover, for all $\by\in H$ we find \[(\by-\by^*)\perp \ker(\bA^T)=\Span(\bn_1,\dotsc,\bn_k)\] since $\bN(\by-\by^*)=\bN\by^*-\bN\by^*=\bzero$. Hence $\by^n-\by^* \in (\ker(\bA^T))^\perp=\im(\bA)$ for all $n\in \N_0$.  We now want to show that the last $N-k$ column vectors of the invertible matrix $\bS=(\bv_1 \dotsc \bv_k \bv_{k+1} \dotsc \bv_N)$ of generalized eigenvectors associated with $\bD\bg(\by^*)$, see  \eqref{eq:DG_digaonal}, form a basis of $\im(\bA)$. Since $\bg$ conserves all linear invariants we observe
\begin{equation*}
\begin{aligned}
\bn_i^T\bD\bg(\by^*)\bv&=\lim_{h\to 0}\frac{1}{h}\Bigl(\bn_i^T\bg(\by^*+h\bv)-\bn_i^T\bg(\by^*)\Bigr)=\lim_{h\to 0}\frac{1}{h}\Bigl(\bn_i^T(\by^*+h\bv)-\bn_i^T\by^*\Bigr)=\bn_i^T\bv
\end{aligned}
\end{equation*}
for all $\bv\in \R^N$, and in particular we find
\begin{equation}
\bn_i^T(\bD\bg(\by^*)-\lambda\bI)\bv = \bn_i^T\bD\bg(\by^*)\bv - \lambda\bn_i^T\bv=(1-\lambda)\bn_i^T\bv.\label{eq:n_i(Dg(y*)-lambdaI)v}
\end{equation}
If $\bv$ is a generalized eigenvector of $\bD\bg(\by^*)$ corresponding to an eigenvalue $\lambda\neq 1$, so that \[(\bD\bg(\by^*)-\lambda\bI)^m\bv =\bzero\] is satisfied for some $m\in \N$, it follows from \eqref{eq:n_i(Dg(y*)-lambdaI)v} that
\begin{equation*}0=\bn_i^T(\bD\bg(\by^*)-\lambda\bI)^m\bv=(1-\lambda)\bn_i^T(\bD\bg(\by^*)-\lambda\bI)^{m-1}\bv=(1-\lambda)^m\bn_i^T\bv,
\end{equation*}
which implies $\bn_i^T\bv=0$ as $\lambda\neq 1$. Hence, all generalized eigenvectors $\bv$ corresponding to an eigenvalue $\lambda\neq 1$ are elements of $(\ker(\bA^T))^\perp=\im(\bA)$. Now note that $\bv_{k+1},\dotsc,\bv_N$ are $N-k$ generalized eigenvectors corresponding to eigenvalues of absolute value smaller than 1. Finally, since \[\dim(\im(\bA))=N-\dim(\ker(\bA))=N-k,\]
the vectors $\bv_{k+1},\dotsc,\bv_N$ form a basis of $\im(\bA)$. Since $\by^n-\by^*\in \im(\bA)=\Span(\bv_{k+1},\dotsc,\bv_N)$ there exist coefficients $\gamma^n_i\in \R$ such that for all $n\in \N_0$ we can write
\begin{align}
\by^n-\by^*=\sum_{i=k+1}^N \gamma^n_i\bv_i.\label{eq:yn-y*}
\end{align}

In order to prove the local convergence of the iterates $\by^n$ to $\by^*$ we  investigate the local convergence of $\bw^n$ to the origin. According to Theorem~\ref{Thm:Ex_CM}~\ref{item:attractivity}
the distance of the iterates $\bw^n\in\R^N$ from \ref{item:stability} to the center manifold given in \eqref{eq:CM} tends to zero for $n\to \infty$, if the iterates stay within a certain neighborhood of the origin. More precisely, this means that the sequence  $(\bw^n)_{n\in \N_0}$ approaches \[\{(\bw_1,\bw_2)\in \R^N\mid \norm{ \bw_1}<\epsilon,\bw_2=\bzero\} = \Span(\be_1,\dotsc,\be_k)\cap B_\epsilon(\bzero)\] for $n\to \infty$, if $\norm{ \bw^n}<\epsilon$ for $i=1,\dotsc, N$ and all $n\in \N_0$, where $\epsilon>0$ is sufficiently small. Now, since the origin is a stable fixed point of $\bG$, as shown in \ref{item:stability}, there exists $\widetilde{\delta}>0$ such that $\norm{\bw^0}<\widetilde{\delta}$ implies $\norm{\bw^n}<\epsilon$ for all $n \in\N_0$. Assuming $\norm{\bw^0}<\widetilde{\delta}$, we can conclude 
\begin{equation}\label{eq:limwn}
\lim_{n\to\infty}\bw^n\in \Span(\be_1,\dotsc,\be_k).
\end{equation}
Furthermore, from \eqref{eq:yn-y*} it follows
\begin{equation*}
\begin{aligned}
\bw^n=\bT(\by^n)&=\bS^{-1}(\by^n-\by^*)=\bS^{-1}\biggl(\sum_{i=k+1}^N \gamma^n_i\bv_i\biggr)=\sum_{i=k+1}^N \gamma^n_i\bS^{-1}\bv_i=\sum_{i=k+1}^N\gamma^n_i\be_i.
\end{aligned}
\end{equation*}
In particular this means $\bw^n\in \Span(\be_{k+1},\dots,\be_N)$, and hence, in combination with \eqref{eq:limwn} one obtains \[\lim_{n\to\infty}\bw^n\in \Span(\be_1,\dotsc,\be_k)\cap\Span(\be_{k+1},\dotsc,\be_N)=\{\bzero\},\] i.\,e.\ $\lim_{n\to\infty}\bw^n=\bzero$.  Due to the transformation $\bT$ this is equivalent to $\lim_{n\to\infty}\by^n=\by^*$ for $\by^0\in H\cap D$ satisfying $\norm{\by^0-\by^*}<\delta=\widetilde{\delta}/\norm{\bS^{-1}}$ since then \[\norm{\bw^0}=\norm{\bT(\by^0)}=\norm{\bS^{-1}(\by^0-\by^*)}\leq \norm{\bS^{-1}}\norm{\by^0-\by^*}<\widetilde{\delta}\] follows.\qed
\end{enumerate}
\end{proof}

\section{Stability of MPRK22($\alpha$) Schemes}\label{sec:stab_mprk22}
As a next step we follow the approach of \cite{IKM2122} to analyze MPRK22($\alpha$) schemes applied to \eqref{eq:Dahlquist_System} with positive steady state solutions $\by^*>\bzero$. First, we write the schemes as $\bM(\by^n)\by^{n+1}=\by^n$, which can be achieved as follows. Since we are focusing on the MPRK22 schemes 
from \eqref{eq:MPRK22b}, we define the matrices 
\begin{align}
\bB&=(\bI-\Delta t\alpha\bA)^{-1},\label{eq:B}\\
\bC&=\left(1-\frac{1}{2\alpha}\right)\bI+\frac{1}{2\alpha}\bB,\label{eq:C}
\end{align}
with the identity matrix $\bI\in \R^{N\times N}$ and system matrix $\bA$ from \eqref{eq:Dahlquist_System}. Note, that these matrices coincide with $\bB_\gamma$ and $\bC_\gamma$ for $\gamma=1$ and $N=2$ from \cite{IKM2122}. Next, we introduce the functions
\begin{align}
\bssigma_i:\R^N_{>0}\to\R^N_{>0}, \quad \bssigma_i(\by^n)&=(\bB\by)_i^{\frac{1}{\alpha}}(y_i)^{1-\frac{1}{\alpha}}, i=1,\dotsc, N,\label{eq:sigma}\\
\bm \tau_i:\R^N_{>0}\to\R^N_{>0}, \quad\bm \tau(\by^n)_i&=\frac{(\bC\by)_i}{\bssigma_i(\by)}, i=1,\dotsc, N,\label{eq:tau}\\
\bM:\R^N_{>0}\to\R^{N\times N},\quad \bM(\by)&=\bI-\Delta t\bA\diag\bigl(\bm \tau(\by)),\label{eq:M}
\end{align}
which are the straight forward extensions of the corresponding functions in \cite{IKM2122} to $N$ dimensions. With these notations we conclude from \cite[Proposition 3.1, Remark 3.2]{IKM2122}, that the MPRK22($\alpha)$ schemes applied to a conservative system \eqref{eq:Dahlquist_System} can be represented by
\begin{equation}\label{eq:Myn+1=yn}
\bM(\by^n)\by^{n+1}=\by^n.
\end{equation}
In particular, if $\bn\in \ker(\bA^T)$, then 
\[\bn^T\by^n=\bn^T\bM(\by^n)\by^{n+1}=\bn^T\by^{n+1}-\Delta t\bn^T\bA\diag(\bm \tau(\by))=\bn^T\by^{n+1},\]
which means that the MPRK22($\alpha$) schemes unconditionally conserve all linear invariants of the linear test equation.

We also want to point out that for $\by^*\in \ker(\bA)\cap \R^N_{>0}$ it follows from \eqref{eq:B} that $\bB^{-1}\by^*=(\bI-\Delta t\alpha\bA)\by^*=\by^*$, and hence, $\bB\by^*=\by^*$. Thus, from \eqref{eq:C} and \eqref{eq:sigma}, one immediately obtains $\bC\by^*=\by^*$ and $\sigma_i(\by^*)=y^*_i$ for $i=1,\dotsc N$. As a consequence we conclude with \eqref{eq:tau} that $\diag(\bm \tau(\by^*))=\bI.$ Altogether,  for $\by^*\in \ker(\bA)\cap \R^N_{>0}$, the relations
\begin{equation}
\bB\by^*=\by^*,\quad \bC\by^*=\by^*,\quad \sigma_i(\by^*)=y^*_i \text{ for $i=1,\dotsc N$}\quad\text{and}\quad \diag(\bm \tau(\by^*))=\bI\label{eq:diag(tau(y*))=I}
\end{equation}
hold.
Next, taking into account \eqref{eq:M} and \eqref{eq:diag(tau(y*))=I} leads to 
\begin{equation}
\bM(\by^*)=\bI-\Delta t\bA\label{eq:M(y*)},
\end{equation}
and hence, $\bM(\by^*)\by^*=\by^*$, from which it follows by  \eqref{eq:Myn+1=yn} that $\by^n=\by^*$ implies $\by^{n+1}=\by^*$, i.\,e.\ any $\by^*\in  \ker(\bA)\cap \R^N_{>0}$ is a fixed point of the MPRK22($\alpha$) schemes.

In \cite[Proposition 3.1]{IKM2122}, a function $\bg:\R^N_{>0}\to \R^N_{>0}$ was constructed such that $\by^{n+1}=\bg(\by^n)$ holds true for $N=2$. Then the Jacobian of $\bg$ was computed and its eigenvalues were analyzed by using \cite[Theorem 2.9]{IKM2122}, which is the two-dimensional version of Theorem \ref{Thm_MPRK_stabil}. However, the representation given in \cite[Lemma 3.5]{IKM2122} does not hold true in the general case of $N>2$. 
Hence, we follow a different approach in order to compute $\bD\bg(\by^*)$ for some $\by^*\in \ker(\bA)\cap \R^N_{>0}$ and $N\geq 2$. Introducing the map 
\begin{equation}
\widetilde{\bg}:\R^{N}_{>0}\times \R^N_{>0}\to \R^N \text{ with } \widetilde{\bg}(\bx,\by)=\bM(\bx)\by-\bx\label{eq:g_tilde}
\end{equation}
we first observe that $\widetilde{\bg}(\by^n,\by^{n+1})=\bzero$, and we are interested to compute the Jacobian of the function
\begin{equation}
\by=\bg(\bx)=(\bM(\bx))^{-1}\bx=(\bI-\Delta t\bA\diag\bigl(\bm \tau(\bx)))^{-1}\bx.\label{eq:g(x)}
\end{equation}
Note that, as $\by^*$ is a fixed point of the MPRK22($\alpha$) schemes we have $\widetilde{\bg}(\by^*,\by^*)=\bzero$ or, equivalently, $\bg(\by^*)=\by^*$. Furthermore, we know that $\widetilde{\bg}\in~\mathcal{C}^\infty(\R^{N}_{>0}\times \R^N_{>0},\R^N_{>0})$ and $\bg\in \mathcal{C}^\infty(\R^N_{>0},\R^N_{>0})$ as $\bm \tau(\by)\in \mathcal{C}^\infty(\R^N_{>0},\R^N_{>0})$ can be seen along the same lines as in the proof of \cite[Lemma 3.4]{IKM2122}. Choosing a neighborhood $\mathcal D$ of $\by^*$ such that $\overline{\mathcal D}\tm \R^N_{>0}$, we immediately see that the first derivatives of $\bg$ are Lipschitz continuous on $\mathcal{D}$ as they are on the compact set $\overline{\mathcal D}$.

Now let  $\left(\bD_\bx\widetilde{\bg}(\by^*,\by^*)\right)_{ij} =\partial_{x_j} \widetilde{\bg}_i(\by^*,\by^*)$ for $i,j=1,\dotsc,N$ and define analogously the Jacobian  $\bD_\by\widetilde{\bg}(\by^*,\by^*)$ with respect to $\by$. Since 
\begin{equation}
\bD_\by\widetilde{\bg}(\by^*,\by^*)=\bM(\by^*)\label{eq:D_yg(y*,y*)}\overset{\eqref{eq:M(y*)}}{=}I-\Delta t \bA
\end{equation}
is invertible we can make use of the implicit function theorem, which states that 
\begin{equation}
\bD\bg(\by^*)=-(\bD_\by\widetilde{\bg}(\by^*,\by^*))^{-1}(\bD_\bx\widetilde{\bg}(\by^*,\by^*)).\label{eq:Dg(y*)a}
\end{equation}
With \eqref{eq:M} we find \[\bM(\bx)\by=\by-\Delta t\bA \diag(\bm \tau(\bx))\by=\by-\Delta t\bA \diag(\by)\bm \tau(\bx).\]
Hence, plugging this into \eqref{eq:g_tilde} yields
\begin{equation}
\begin{aligned}\label{eq:D_xg(y*,y*)}
\bD_\bx\widetilde{\bg}(\by^*,\by^*)&=\bD_\bx(\bM(\bx)\by)\big|_{(\bx,\by)=(\by^*,\by^*)}-\bI\\
&=\bD_\bx(\by-\Delta t\bA \diag(\by)\bm \tau(\bx))\big|_{(\bx,\by)=(\by^*,\by^*)}-\bI\\&=-\Delta t\bA\diag(\by^*)\bD\bm\tau(\by^*)-\bI.
\end{aligned}
\end{equation}
Following the lines of the proof of \cite[Lemma 3.5]{IKM2122} we similarly get \[\bD\bm\tau(\by^*)=\frac{1}{2\alpha}(\diag(\by^*))^{-1}(\bI-\bB)\] also for $N\geq 2$. Hence, using \eqref{eq:D_xg(y*,y*)} we see
\[	\bD_\bx\widetilde{\bg}(\by^*,\by^*)=-\frac{1}{2\alpha}\Delta t\bA(\bI-\bB)-\bI.\] Using the equation above in combination with \eqref{eq:D_yg(y*,y*)} one can rewrite \eqref{eq:Dg(y*)a} in the form
\begin{equation*}
\bD\bg(\by^*)=(\bI-\Delta t\bA)^{-1}\left(\frac{1}{2\alpha}\Delta t\bA(\bI-\bB)+\bI\right).
\end{equation*}
To summarize these results, we formulate the following proposition.
\begin{prop}\label{Prop:MPRK_Dg(y*)}
Let $\bg\from\R^N_{>0}\to\R^N_{>0}$ be given by the application of MPRK22($\alpha$) to the differential equation \eqref{eq:Dahlquist_System} with $\bm 1\in \ker(\bA^T)$.
Then any $\by^*\in \ker(\bA)\cap \R^N_{>0}$ is a fixed point of $\bg$ and $\bg\in \mathcal{C}^\infty(\R^N_{>0},\R^N_{>0})$, whereby the first derivatives of $\bg$ are Lipschitz continuous in an appropriate neighborhood of $\by^*$. Moreover, all linear invariants are conserved and the Jacobian of $\bg$ satisfies
\begin{equation}
\bD\bg(\by^*)=(\bI-\Delta t\bA)^{-1}\left(\frac{1}{2\alpha}\Delta t\bA\left(\bI-(\bI-\alpha\Delta t\bA)^{-1}\right)+\bI\right).\label{eq:Dg(y*)}
\end{equation}
\end{prop}

By means of a straightforward but excessive calculation one can prove that \eqref{eq:Dg(y*)} coincides with the corresponding expression given in \cite{IKM2122} for the specific case of a linear PDS with two equations.

We want to recall at this point that any eigenvector of $\bA$ with eigenvalue $\lambda$ is an eigenvector of the matrix $\bB=(\bI-\Delta t\alpha \bA)^{-1}$ with eigenvalue $(1-\Delta t\lambda\alpha)^{-1}$, and due to \eqref{eq:Dg(y*)}, also an eigenvector of $\bD\bg(\by^*)$ with corresponding eigenvalue $R(\Delta t\lambda)$, where 
\begin{equation}
R(z)=\frac{\frac{1}{2\alpha}z(1-\frac{1}{1-\alpha z})+1}{1-z}=\frac{\left(\frac{z}{\alpha}(1-\alpha z-1)+2(1-\alpha z)\right)}{2(1-\alpha z)(1-z)}=\frac{-z^2-2\alpha z+2}{2(1-\alpha z)(1-z)}.\label{eq:R(z)}
\end{equation}
We call $R$ the \emph{stability function} of the MPRK22($\alpha$) scheme because if the same analysis is carried out for a Runge--Kutta scheme the function $R$ is the stability function the Runge--Kutta method. Note, that $R$ from \eqref{eq:R(z)} coincides with the stability function $R_1$ of the MPRK22$(\alpha$) scheme as derived in \cite[Lemma 3.6]{IKM2122}. Hence, we can cite \cite[Lemma 3.8]{IKM2122}, i.\,e.\ $\lvert R(z)\rvert<1$ for all $z=\Delta t\lambda\in \R^-$. However, as $\bA\in \R^{N\times N}$ may possess also complex eigenvalues we have to prove that $\lvert R(z)\rvert<1$ holds even for all $z\in \C^-$.
\begin{prop}\label{Prop:R(z)<1}
The stability function $R(z)=\frac{-z^2-2\alpha z+2}{2(1-\alpha z)(1-z)}$ from \eqref{eq:R(z)} with $\alpha> \frac{1}{2}$ satisfies $R(0)=1$ and $\lvert R(z)\rvert<1$ for all $z\in \C^-\setminus\{0\}$. For $\alpha=\frac12$ we have $\lvert R(z)\rvert<1$ for all $z$ with $\operatorname{Re}(z)<0$ and $\lvert R(z)\rvert =1$, if $\operatorname{Re}(z)=0$. 
\end{prop}
\begin{proof}
We first investigate $\lvert R(z)\rvert$ for $z=\ii y$ and $y\in \R$. 
A small calculation reveals that the numerator of $\lvert R(z)\rvert$ can be written as
\begin{equation}\label{eq:numerator_R(z)}
\lvert -z^2-2\alpha z+2\rvert^2=\lvert y^2+2+(-2\alpha y)\ii\rvert^2=(y^2+2)^2+4\alpha^2 y^2=y^4+4y^2(1+\alpha^2)+4.
\end{equation}
Performing a similar calculation for the denominator of $\lvert R(z)\rvert$ we find
\begin{equation*}
\begin{aligned}
\lvert 2(1-\alpha z)(1-z)\rvert^2&=\lvert 2\alpha z^2-2z(1+\alpha) +2\rvert^2=\lvert -2\alpha y^2+2 +(- 2y(1+\alpha))\ii\rvert^2\\
&=(-2\alpha y^2+2)^2+4y^2(1+\alpha)^2						=4\alpha^2y^4+4y^2(1+\alpha^2)+4.
\end{aligned}
\end{equation*}
Using \eqref{eq:numerator_R(z)} and $\alpha= \frac12$ we see that $\lvert R(z)\rvert =  1$ on the imaginary axis, and if $\alpha>\frac12$ we find $\lvert R(z)\rvert< 1$ for all $y\neq 0$.

Next we note that $R$ is a holomorphic function which is defined for all $z\in \C^-$. Since $R$ is rational we can apply the Phragmén-Lindelöf principle on the union of the origin and the interior of $\C^-$ and conclude that $\lvert R(z)\rvert\leq 1$ for all $z\in \C^-$. Furthermore, since $R$ is not constant, we conclude from the maximum modulus principle that there exist no $z_0$ in the interior of $\C^-$ with $\abs{R(z_0)}=1$, or equivalently, $\abs{R(z_0)}<1$ holds for all $z_0$ with $\operatorname{Re}(z_0)<0$.
\end{proof}

As a direct consequence of the application of Theorem \ref{Thm_MPRK_stabil} in combination with the Propositions \ref{Prop:MPRK_Dg(y*)} and \ref{Prop:R(z)<1} we obtain the following two corollaries.

\begin{cor}\label{Cor:MPRKstab}
Let $\by^*$ be a positive steady state of the differential equation \eqref{eq:Dahlquist_System}. Then $\by^*$ is a stable fixed point of the MPRK22($\alpha$) scheme for all $\Delta t>0$, if $\alpha > \frac12$ or $\alpha=\frac{1}{2}$ and all nonzero eigenvalues of $\bD\bg(\by^*)$ given in \eqref{eq:Dg(y*)} have a negative real part.
\end{cor}

\begin{cor}\label{Cor:MPRKstab1}
Let the unique steady state $\by^*$ of the initial value problem \eqref{eq:Dahlquist_System}, \eqref{eq:IC} be positive. Then the iterates of MPRK22($\alpha$) locally converge towards $\by^*$ for all $\Delta t>0$, if $\alpha>\frac12$ or $\alpha=\frac{1}{2}$ and all nonzero eigenvalues of $\bD\bg(\by^*)$ given in \eqref{eq:Dg(y*)} have a negative real part.
\end{cor}

\section{Numerical Experiments}\label{Sec:Num_Tests}
In this section we consider three linear positive and conservative PDS in order to verify the stability properties of the MPRK22($\alpha$) schemes as stated in Corollary \ref{Cor:MPRKstab}. Since all systems are conservative, $\lambda=0$ is an eigenvalue of each system matrix. The test problems are chosen in such a way that the nonzero eigenvalues either lie in $\R^-$ or in $\C^-\setminus\R^-$. Moreover, also the case of an eigenvalue $\lambda=0$ of multiplicity greater than 1 is investigated.

\subsection{Test problem with exclusively real eigenvalues}
We consider the initial value problem
\begin{equation}
\by'=100\Vec{-2& 1 &1\\1 &-4 &1\\1 &3 &-2}\by,\quad  \by(0)=\Vec{1\\ 9\\5}.\label{eq:initProbReal}
\end{equation}
It is easily seen that the only linear invariant is $\mathbf{1}^T\by$. Since the system matrix is a Metzler matrix the exact solution is positive for each positive initial condition.
The eigenvalues of the system matrix are $\lambda_1=0$, $\lambda_2=-300$ and $\lambda_3=-500$, so that the exact solution can be written as
\begin{equation}
\by(t)=c_1\begin{pmatrix}5\\ 3\\ 7\end{pmatrix}+c_2e^{-300t}\begin{pmatrix}-1\\ 0\\ 1\end{pmatrix}+c_3e^{-500t}\begin{pmatrix}0\\ -1\\ 1\end{pmatrix}\label{eq:exsolReal}
\end{equation}
with $c_1=1$, $c_2=4$ and $c_3=-6$.
In addition to the time-dependent history of the solution, Figure \ref{Fig:initProbReal} also shows the rapid convergence to the steady state solution 
\begin{equation*}
\by^*=\lim_{t\to\infty}\by(t)=\Vec{5\\3\\7}
\end{equation*}
as well as the conservativity of the system of differential equations.

\begin{figure}[h!]
\centering
\includegraphics[width=0.5\textwidth]{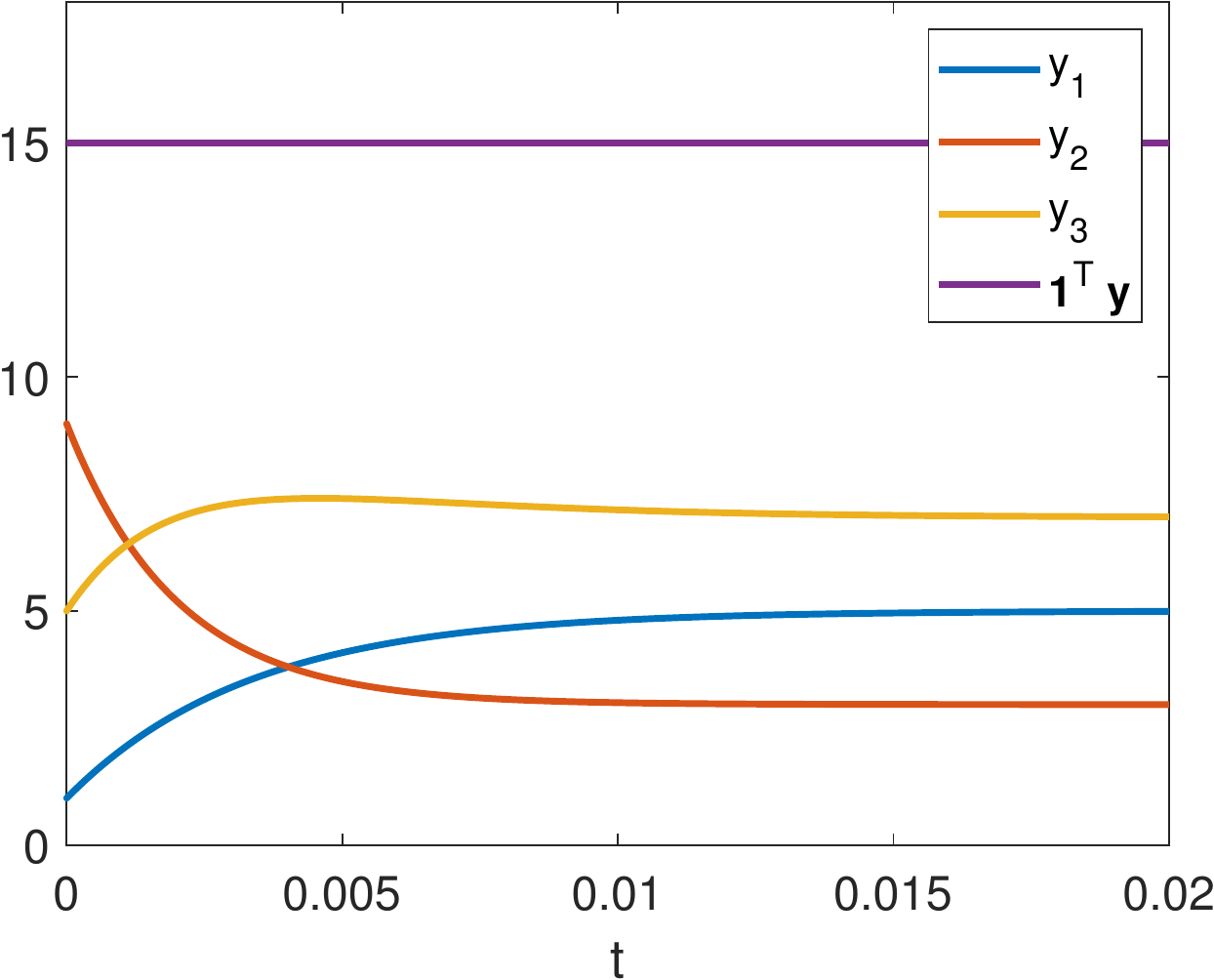}
\caption{Exact solution \eqref{eq:exsolReal} of the initial value problem \eqref{eq:initProbReal} and the linear invariant $\bm 1^T\by$.}\label{Fig:initProbReal}
\end{figure}

\subsection{Test problem with complex eigenvalues}
As a second test case, let us consider a conservative system with complex eigenvalues, namely
\begin{equation}
\by'=100\Vec{-4& 3 &1\\2 &-4 &3\\2 &1 &-4}\by,\quad  \by(0)=\Vec{9\\ 20\\8}\label{eq:initProbIm},
\end{equation}
which again includes a Metzler matrix.
Since $\bm 1^T\by$ is a linear invariant, $\lambda_1=0$ is an eigenvalue of the system matrix. The other two eigenvalues are complex and given by $\lambda_2=100(-6+\ii)$ and $\lambda_3=\overline{\lambda_2}$. Hence, the exact solution contains terms of $\sin$ and $\cos$ and can be written as
\begin{equation}
\begin{aligned}
\by(t)=&c_1\begin{pmatrix}13\\ 14\\ 10\end{pmatrix}+c_2e^{-600t}\left(\cos \left(100t\right)\begin{pmatrix}-1\\ 0\\ 1\end{pmatrix}-\sin \left(100t\right)\begin{pmatrix}1\\ -1\\ 0\end{pmatrix}\right)\\&+c_3e^{-600t}\left(\cos \left(100t\right)\begin{pmatrix}1\\ -1\\ 0\end{pmatrix}+\sin \left(100t\right)\begin{pmatrix}-1\\ 0\\ 1\end{pmatrix}\right).\label{eq:exsolIm}
\end{aligned}
\end{equation}
From the initial condition we find $c_1=1,$ $c_2=-2$ and $c_3=-6$, so that the steady state solution is given by
\begin{equation*}
\by^*=\lim_{t\to\infty}\by(t)=\Vec{13\\14\\10}.
\end{equation*} 
Altogether, Figure \ref{Fig:initProbIm} pictures the exact solution \eqref{eq:exsolIm} and its convergence towards $\by^*$ together with the linear invariant $\bm 1^T\by$.
\begin{figure}[h!]
\centering
\includegraphics[width=0.5\textwidth]{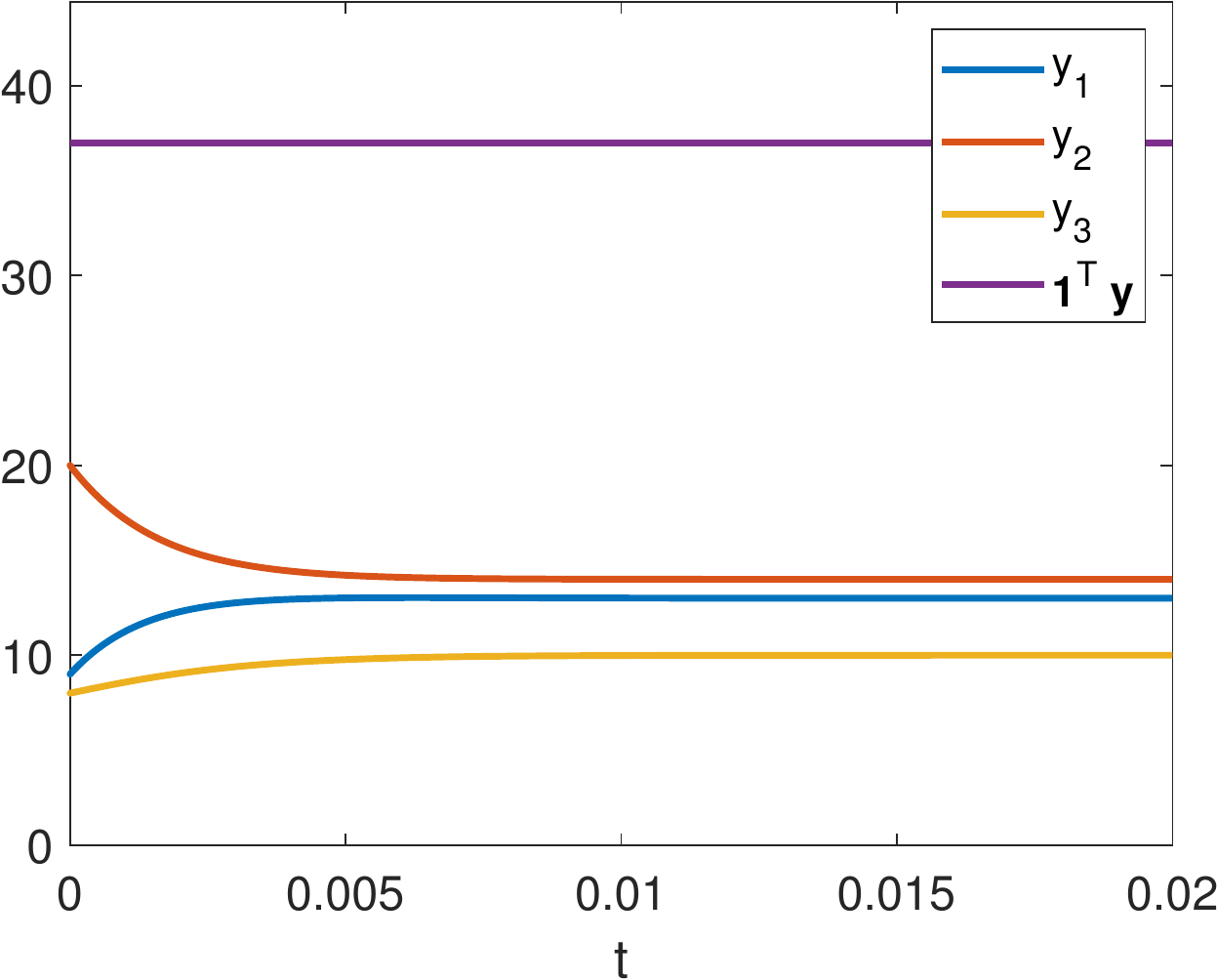}
\caption{The exact solution \eqref{eq:exsolIm} of the initial value problem \eqref{eq:initProbIm} and the linear invariant $\bm 1^T\by$.}\label{Fig:initProbIm}
\end{figure}

\subsection{Test problem with double zero eigenvalue}
The last linear test problem 
\begin{equation}
\by'=100\Vec{-2& 0 &0 &1\\0 &-4 &3& 0\\0 &4& -3 &0\\ 2 & 0&0&-1}\by,\quad  \by(0)=\Vec{4\\ 1\\9\\1}\label{eq:initProb4dim}
\end{equation}
has more than one linear invariant, and thus, the associated steady state solution is an element of a two-dimensional linear subspace. The two eigenvalues $\lambda_1,\lambda_2=0$ of the system matrix, which again is a Metzler matrix, are associated with two linear invariants $\bm 1^T\by$ and $\bn^T\by$ with $\bn=(1,2,2,1)^T$. The remaining eigenvalues are real and given by  $\lambda_3=-300$ and $\lambda_3=-700$. In total, the exact solution is
\begin{equation}
\by(t)=c_1\Vec{0\\1\\ \frac43\\0}+c_2\Vec{1\\0\\0\\2}+c_3e^{-700t}\Vec{0\\1\\-1\\0}+c_4e^{-300t}\Vec{1\\0\\ 0\\-1},\label{eq:exsol4dim}
\end{equation}
where we conclude from the initial condition that
\begin{align*}
c_1=\frac{30}{7},\quad c_2=\frac53, \quad c_3=-\frac{23}{7}\qta c_4=\frac73.
\end{align*}
Plugging the values of $c_1$ and $c_2$ into \eqref{eq:exsol4dim} we find the steady state 
\begin{equation*}
\by^*=\lim_{t\to\infty}\by(t)=\left(\frac53,\frac{30}{7},\frac{40}{7},\frac{10}{3}\right)^T.
\end{equation*}
In total, the exact solution and its asymptotic behavior as well as the two linear invariants are depicted in Figure \ref{Fig:initProb4dim}.
\begin{figure}[h!]
\centering
\includegraphics[width=0.5\textwidth]{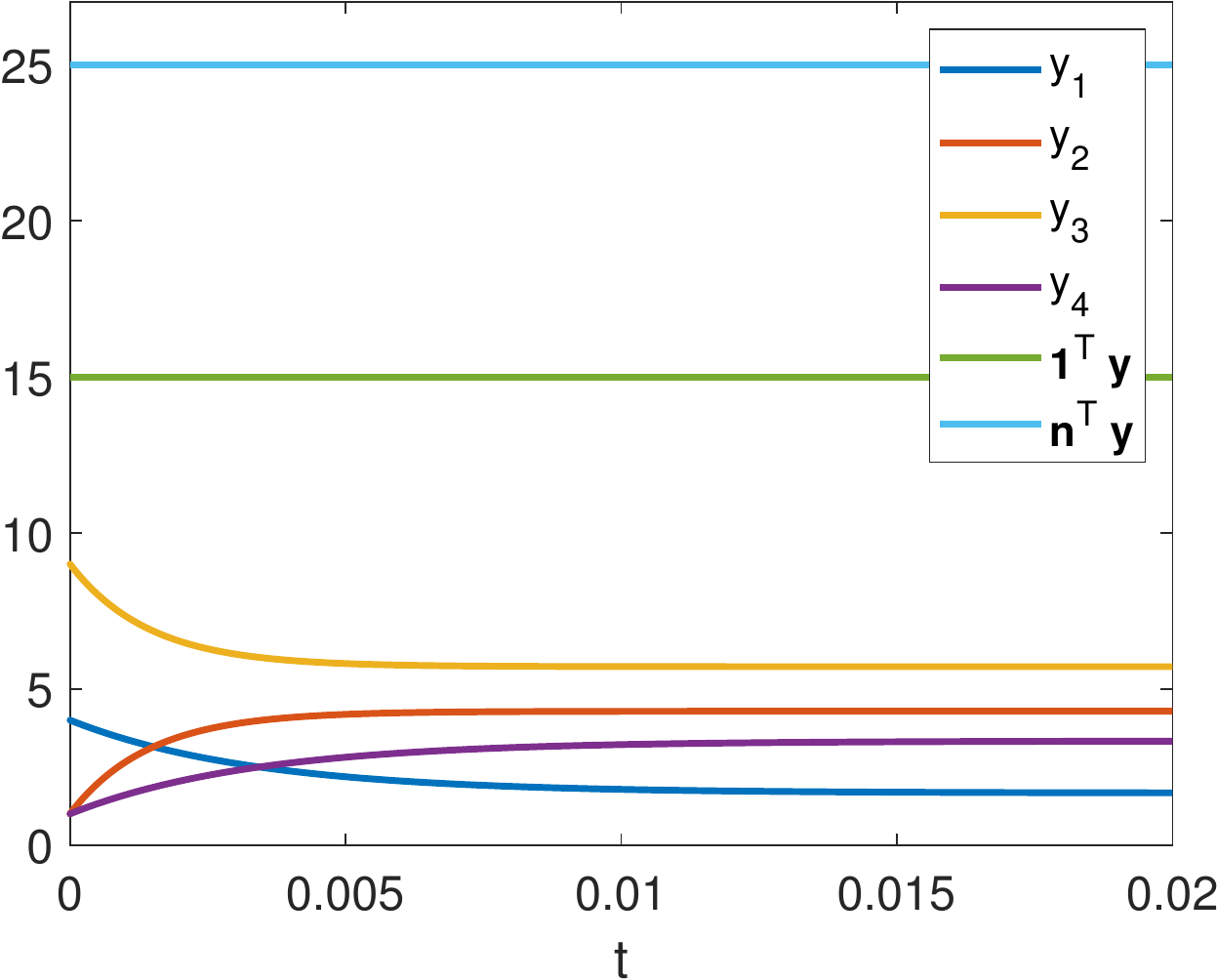}
\caption{The exact solution \eqref{eq:exsol4dim} of the initial value problem \eqref{eq:initProb4dim} and the associated two linear invariants $\bm 1^T\by$ and $\bn^T\by$.}\label{Fig:initProb4dim}
\end{figure}

It is worth to mention that the three test cases represent stiff problems due to the large absolute values of the corresponding eigenvalues. As a result, the exact solution converges very fast to the steady state and has already at time $t=0.02$ a distance $\norm{\by(t)-\by^*}$ to the steady state which is smaller than $2\cdot 10^{-2}$. In order to confirm numerically that MPRK22($\alpha$) schemes are stable as claimed in Corollary \ref{Cor:MPRKstab}, and to demonstrate the local convergence to the steady state solution as stated in Corollary \ref{Cor:MPRKstab1}, we choose a comparably large time step size of $\Delta t=5$ for all examples. 

In the Figures \ref{Fig:MPRKinitProbReal}, \ref{Fig:MPRKinitProbIm} and \ref{Fig:MPRKinitProb4dim}, we compare the MPRK22($\alpha$) schemes for $\alpha\in\left\{\frac12,1,5\right\}$ and find that for $\alpha=1$ or $\alpha=5$, plotted in the respective upper two subfigures, the iterates satisfy  $\norm{\by(t)-\by^*}<3\cdot 10^{-2}$ after $t=40$, i.\,e.\ after $8$ iterations, whereas for $\alpha=\frac12$ we cannot observe the convergence of the iterates towards $\by^*$ within $t\in[0,40]$ as one can see from the bottom left plot in each of the Figures \ref{Fig:MPRKinitProbReal}, \ref{Fig:MPRKinitProbIm} and \ref{Fig:MPRKinitProb4dim}. Nevertheless, the results depicted on the bottom right show that even for the case $\alpha=\frac12$, the stability and convergence proved in Corollaries \ref{Cor:MPRKstab} and \ref{Cor:MPRKstab1} can be confirmed numerically by extending the observation period. Thereby, concerning each model problem the iterates obtained by MPRK($\frac12$) satisfy $\norm{\by(t)-\by^*}<7\cdot 10^{-2}$ at time $t=2\cdot 10^4$.

\begin{figure}[!h]
\begin{subfigure}[t]{0.495\textwidth}
\includegraphics[width=\textwidth]{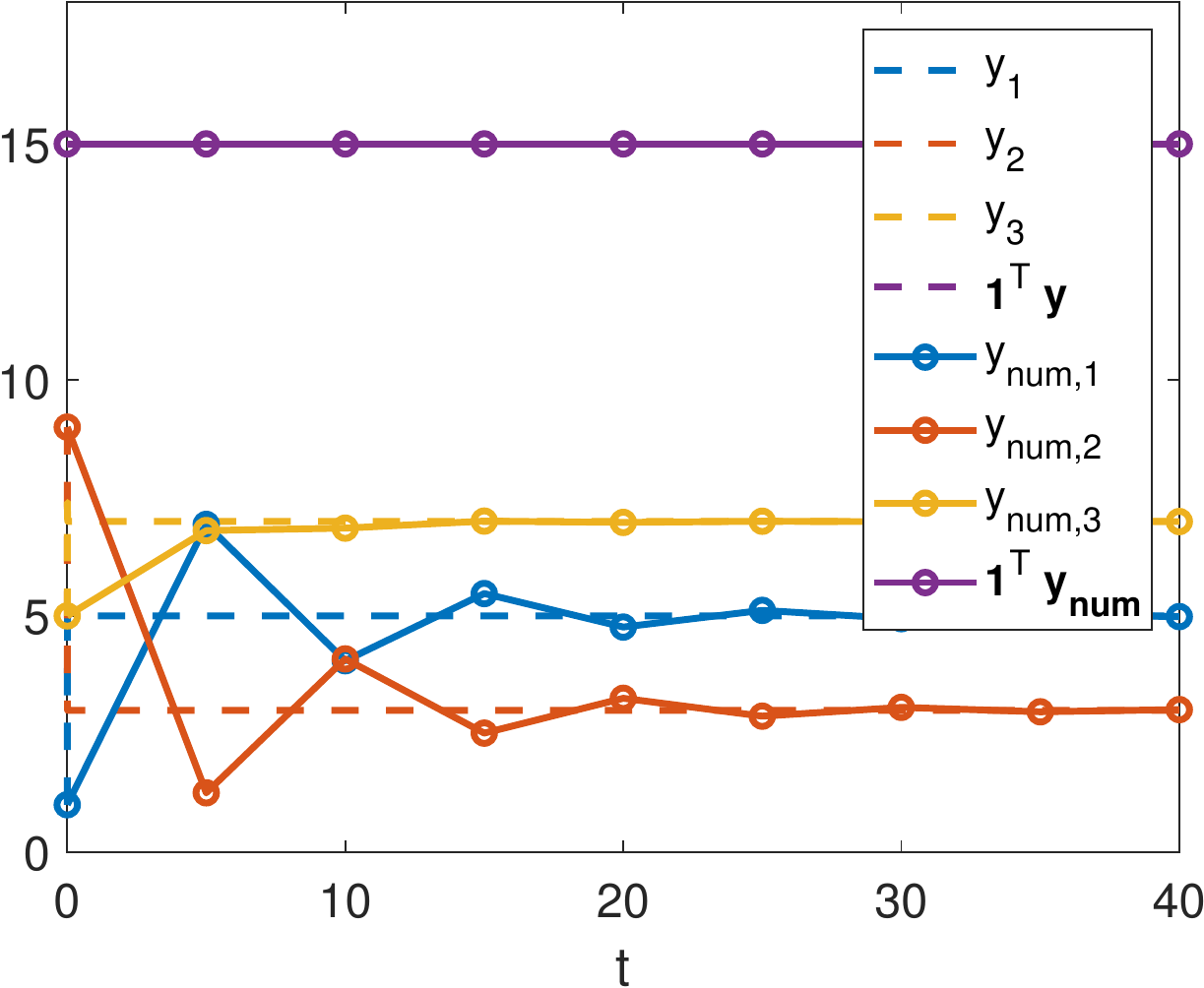}
\subcaption{$\alpha=1$}
\end{subfigure}
\begin{subfigure}[t]{0.495\textwidth}
\includegraphics[width=\textwidth]{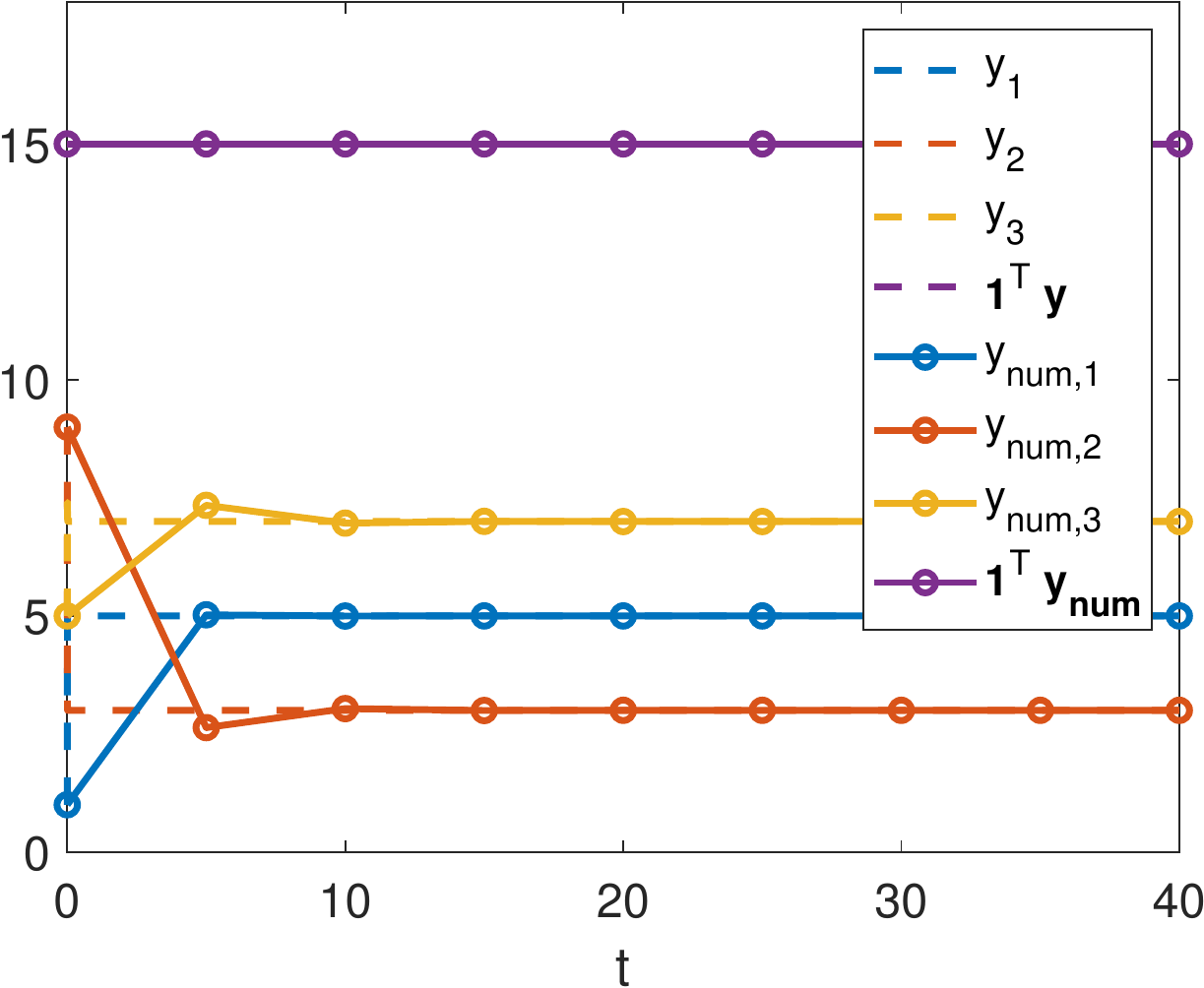}
\subcaption{$\alpha=5$}
\end{subfigure}\\
\begin{subfigure}[t]{0.495\textwidth}
\includegraphics[width=\textwidth]{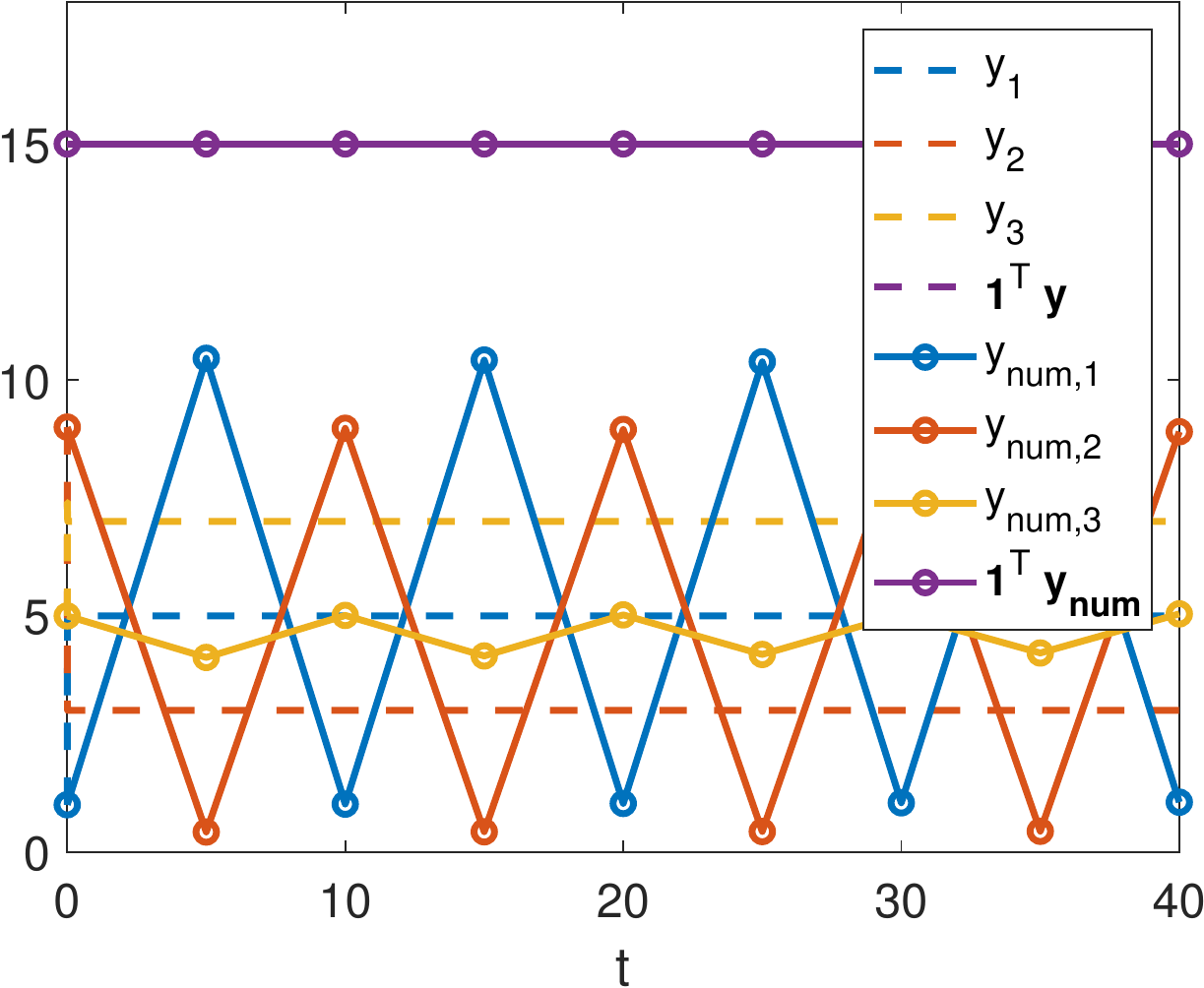}
\subcaption{$\alpha=0.5$}
\end{subfigure}
\begin{subfigure}[t]{0.495\textwidth}
\includegraphics[width=\textwidth]{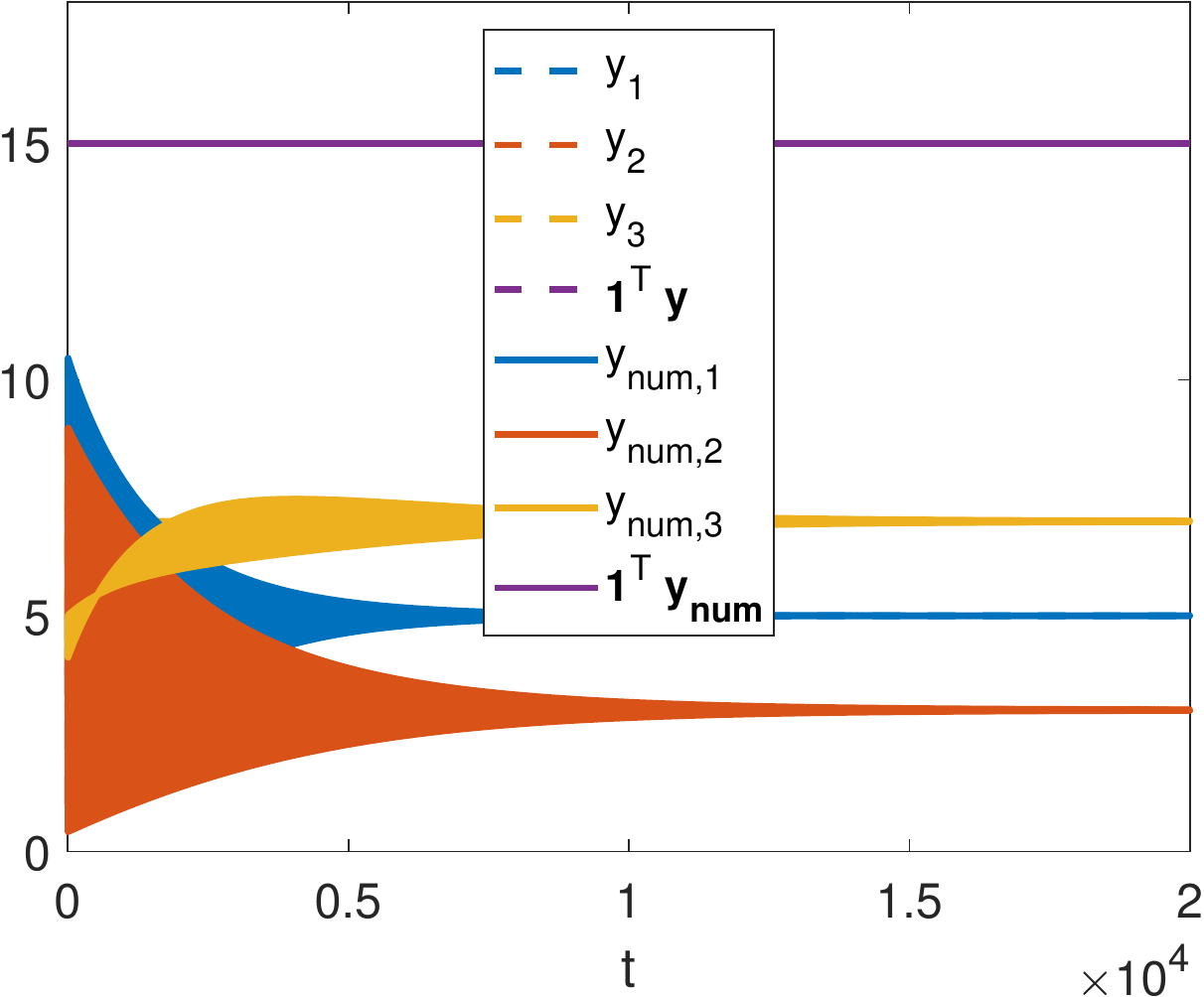}
\subcaption{$\alpha=0.5$}
\end{subfigure}
\caption{Numerical approximations of \eqref{eq:initProbReal} using MPRK22($\alpha$) schemes. The dashed lines indicate the exact solution \eqref{eq:exsolReal}.}\label{Fig:MPRKinitProbReal}
\end{figure}
\begin{figure}[!h]
\begin{subfigure}[t]{0.495\textwidth}
\includegraphics[width=\textwidth]{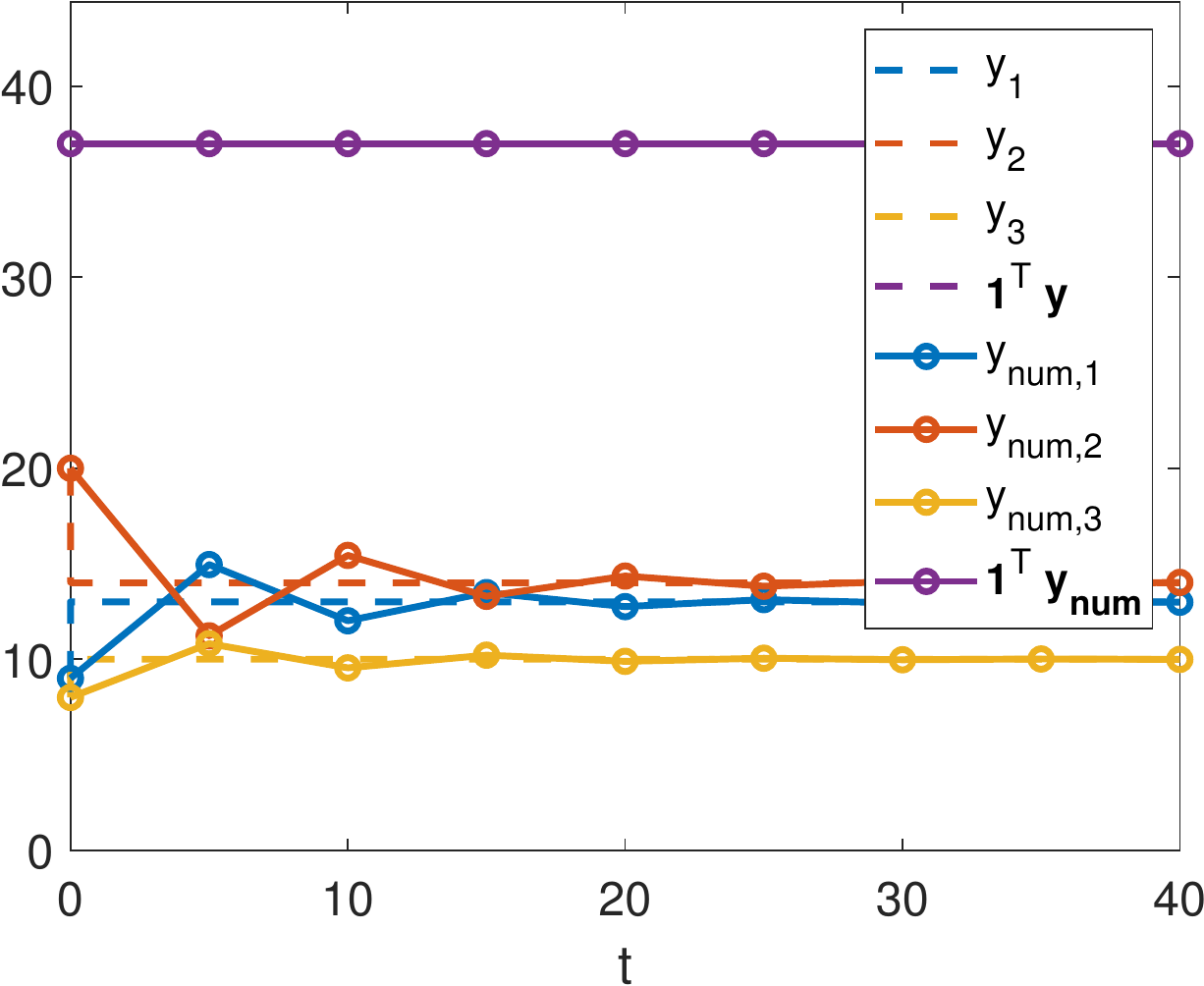}
\subcaption{$\alpha=1$}
\end{subfigure}
\begin{subfigure}[t]{0.495\textwidth}
\includegraphics[width=\textwidth]{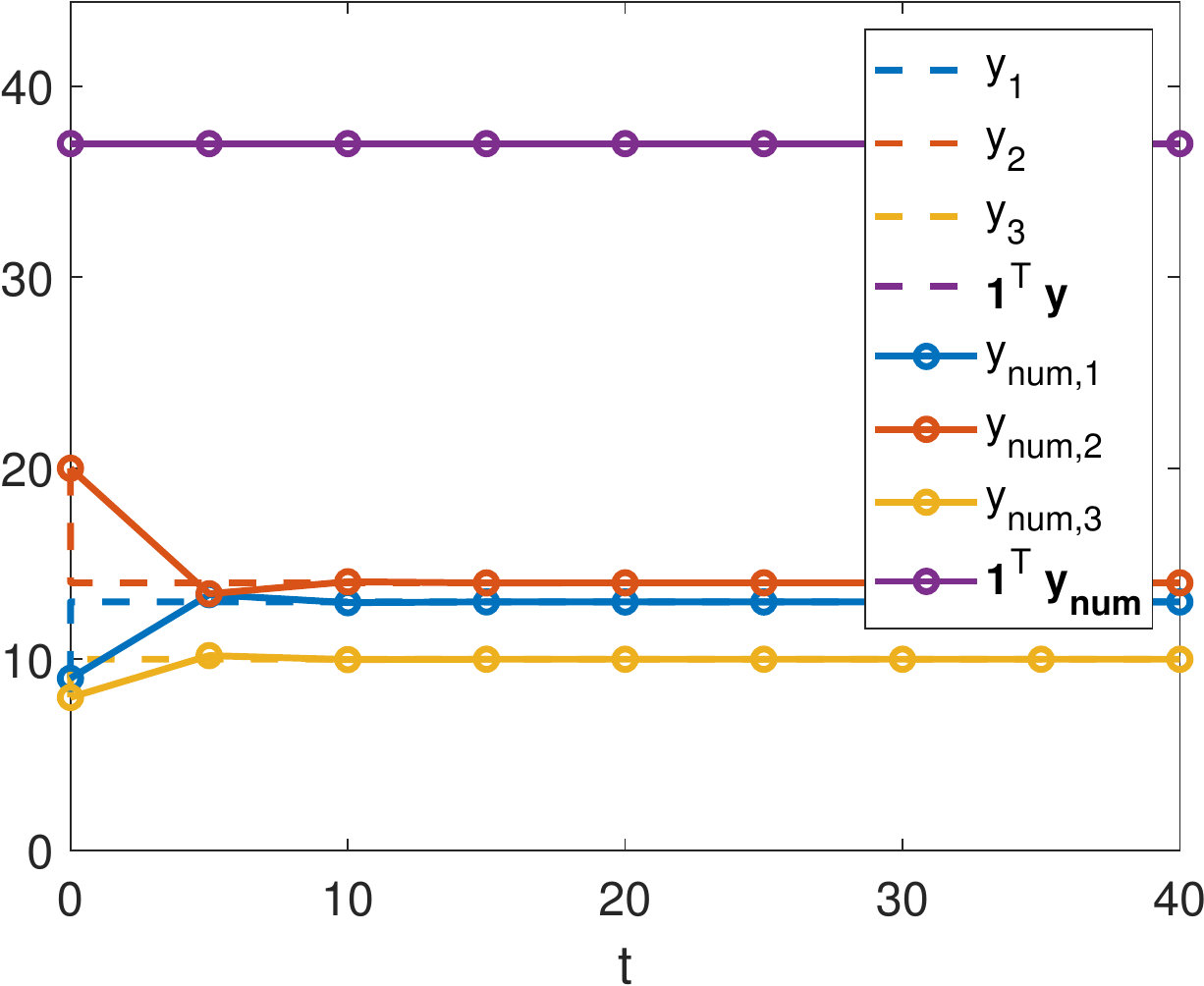}
\subcaption{$\alpha=5$}
\end{subfigure}\\
\begin{subfigure}[t]{0.495\textwidth}
\includegraphics[width=\textwidth]{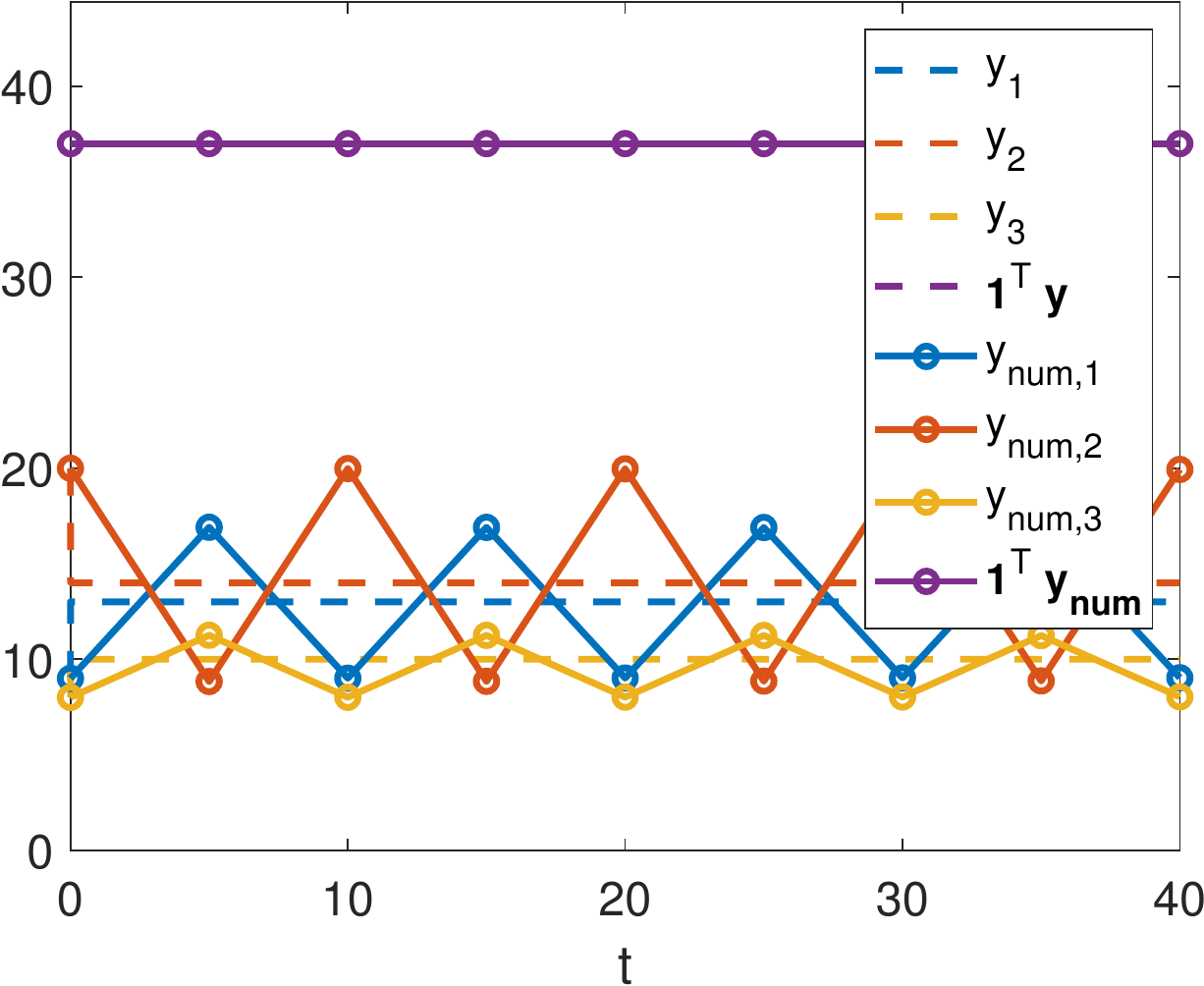}
\subcaption{$\alpha=0.5$}
\end{subfigure}
\begin{subfigure}[t]{0.495\textwidth}
\includegraphics[width=\textwidth]{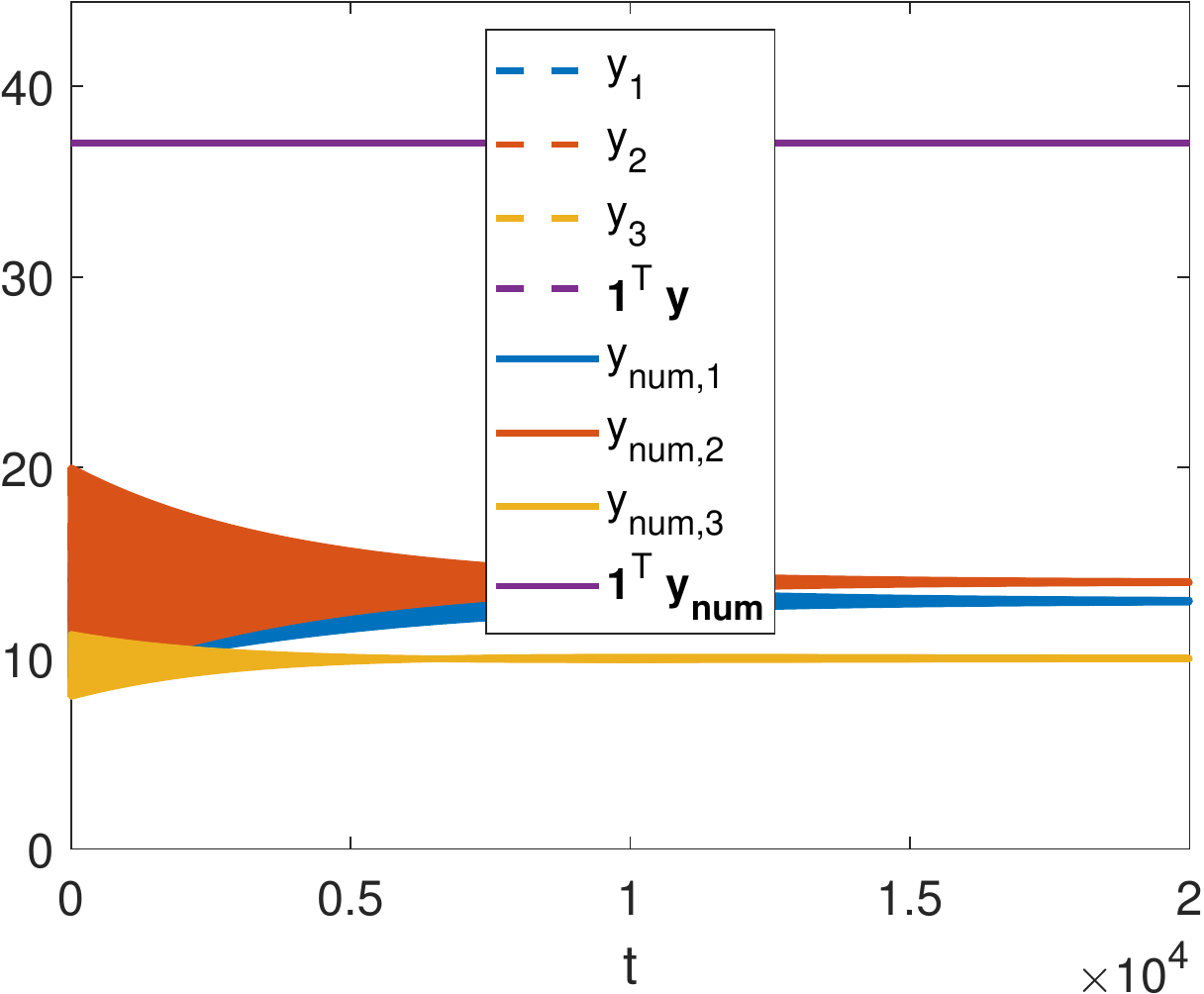}
\subcaption{$\alpha=0.5$}
\end{subfigure}
\caption{Numerical approximations of \eqref{eq:initProbIm} using MPRK22($\alpha$) schemes. The dashed lines indicate the exact solution \eqref{eq:exsolIm}.}\label{Fig:MPRKinitProbIm}
\end{figure}

\begin{figure}[!h]
\begin{subfigure}[t]{0.495\textwidth}
\includegraphics[width=\textwidth]{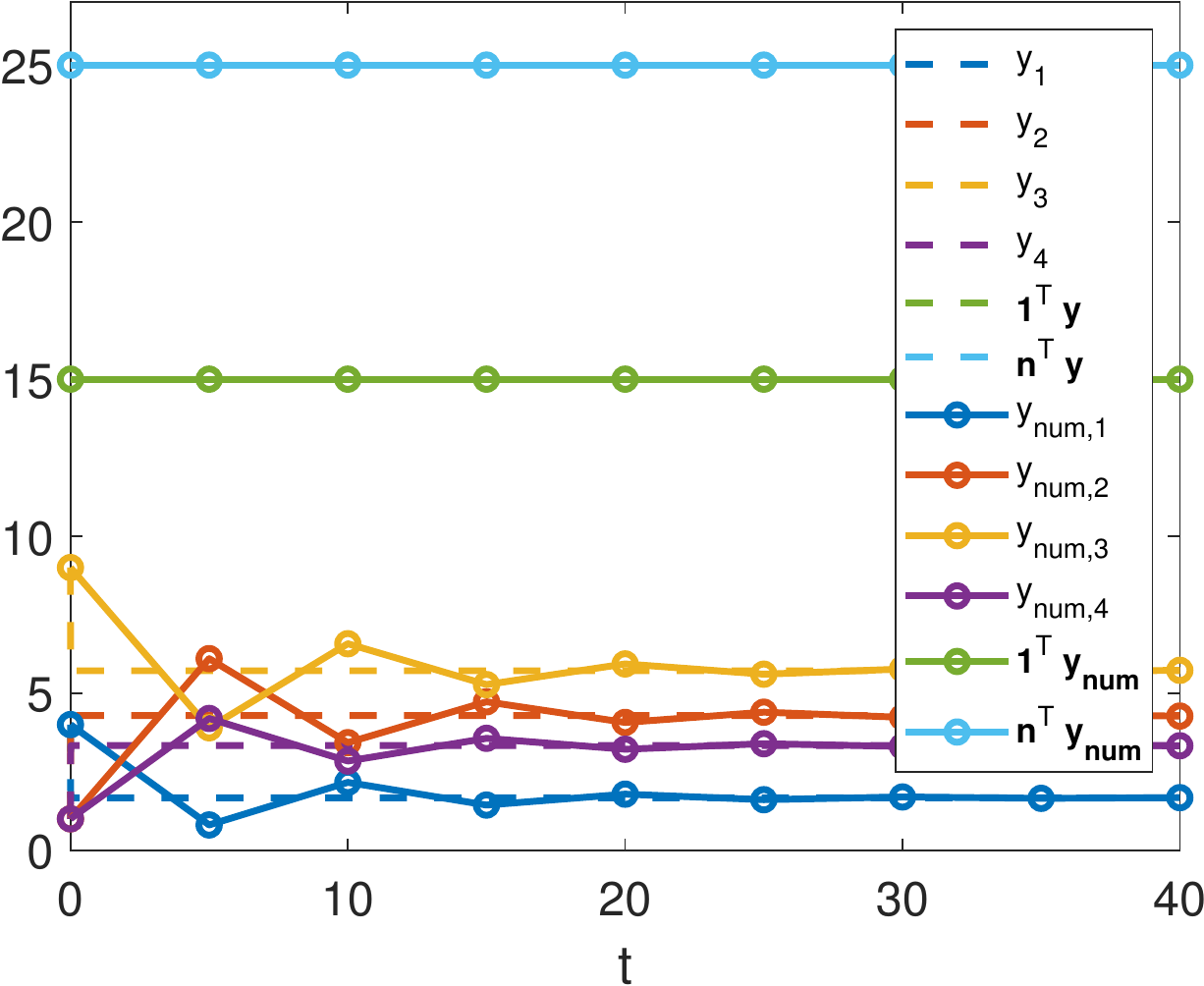}
\subcaption{$\alpha=1$}
\end{subfigure}
\begin{subfigure}[t]{0.495\textwidth}
\includegraphics[width=\textwidth]{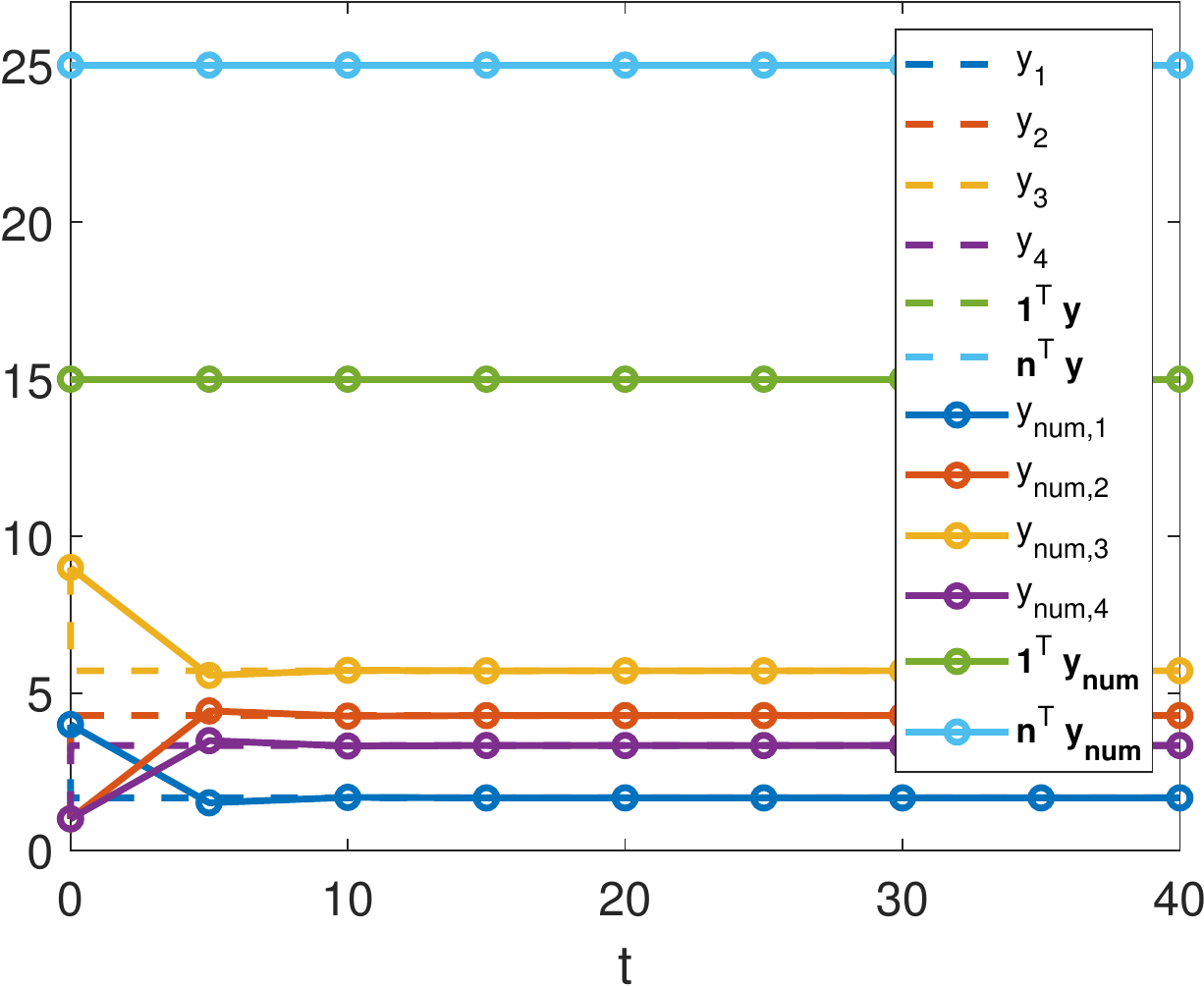}
\subcaption{$\alpha=5$}
\end{subfigure}\\
\begin{subfigure}[t]{0.495\textwidth}
\includegraphics[width=\textwidth]{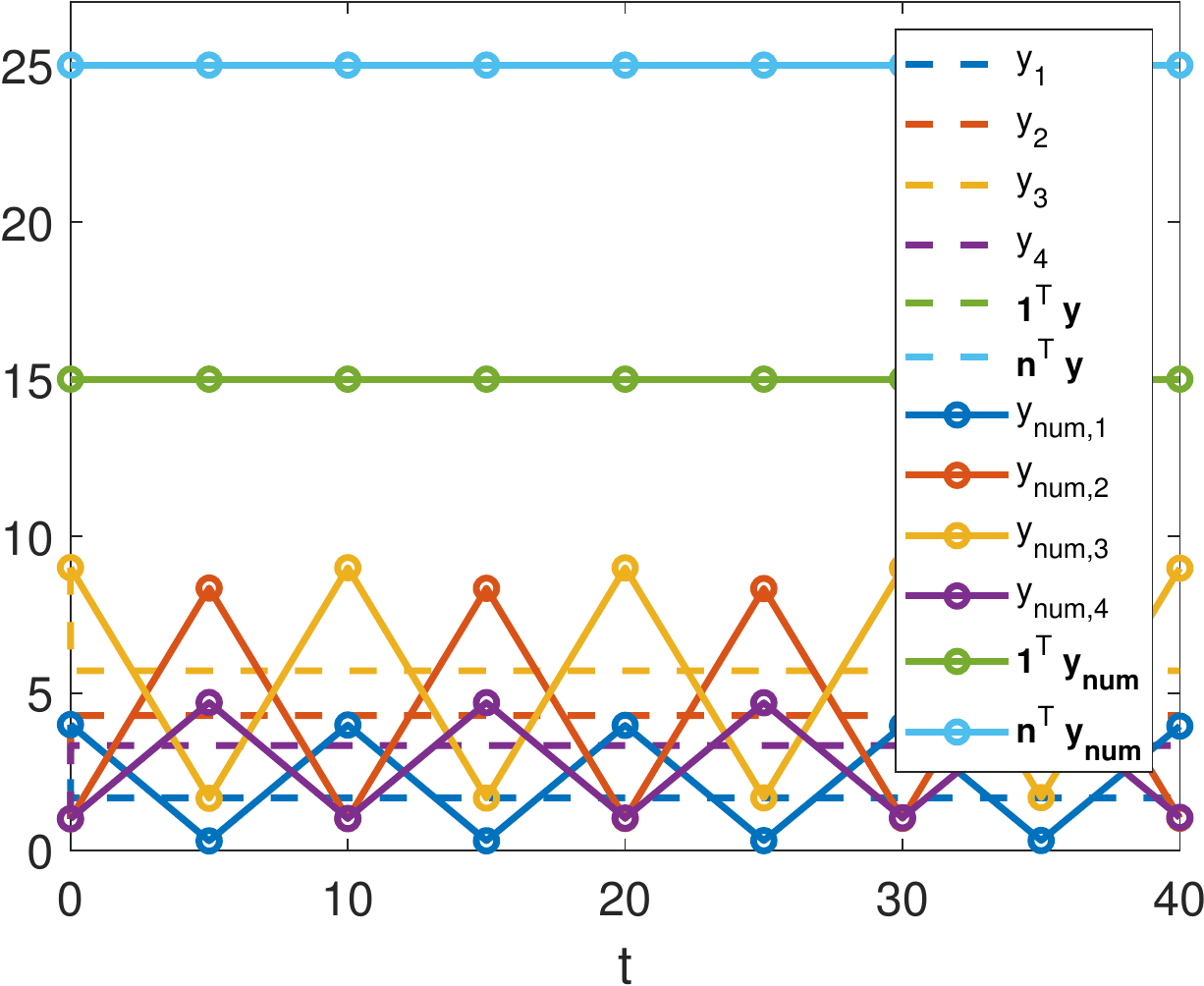}
\subcaption{$\alpha=0.5$}
\end{subfigure}
\begin{subfigure}[t]{0.495\textwidth}
\includegraphics[width=\textwidth]{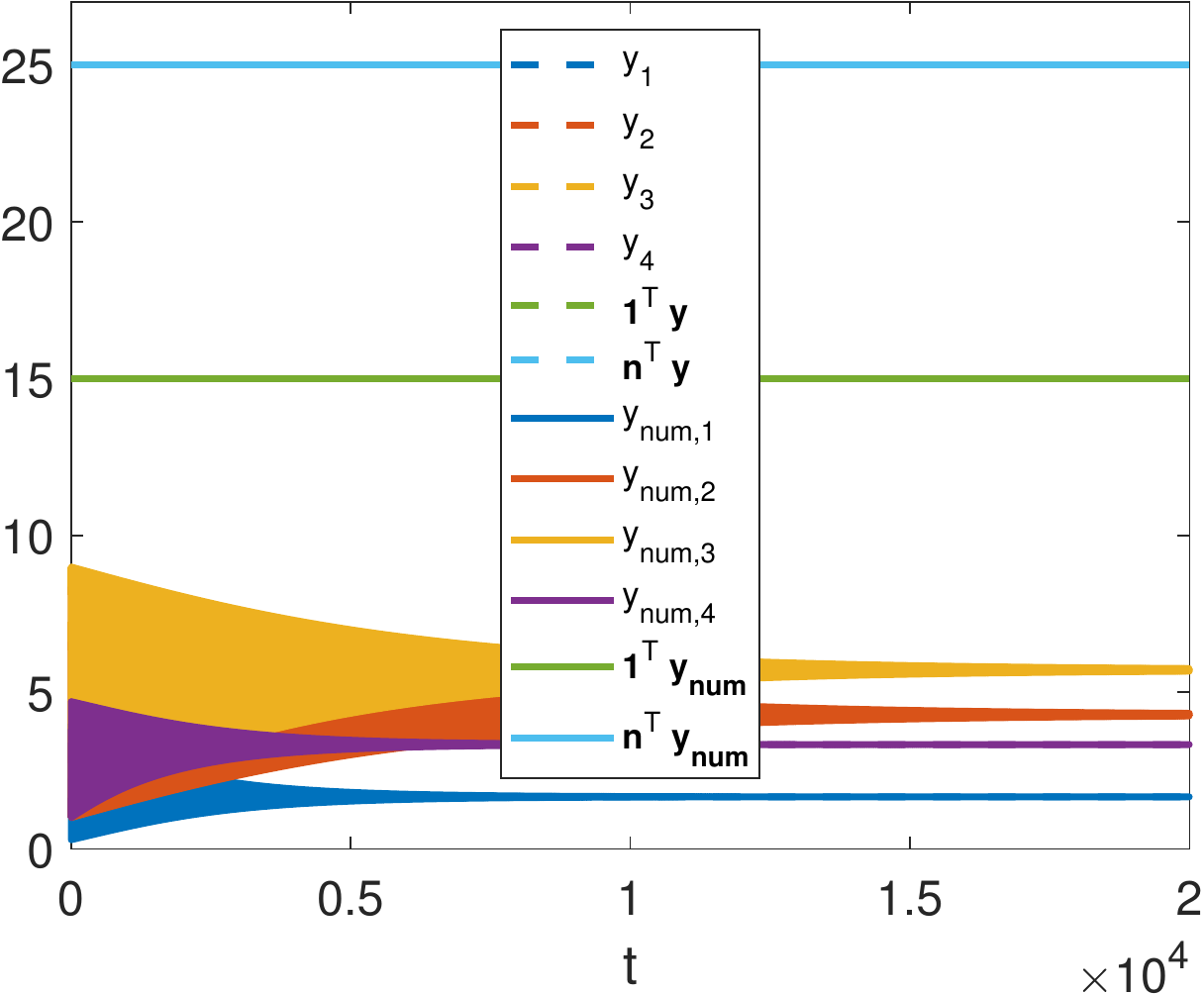}
\subcaption{$\alpha=0.5$}
\end{subfigure}
\caption{Numerical approximations of \eqref{eq:initProb4dim} using MPRK22($\alpha$) schemes. The dashed lines indicate the exact solution \eqref{eq:exsol4dim} and $\bn=(1,2,2,1)^T$.}\label{Fig:MPRKinitProb4dim}
\end{figure}
\section{Summary and Perspectives}
In this paper we generalized the stability analysis from \cite{IKM2122} to an analysis for general time integration schemes conserving at least one linear invariant whereby these schemes do not have to belong to the class of general linear methods. The main result, Theorem \ref{Thm_MPRK_stabil}, gives sufficient conditions for the stability of the methods as well as their local convergence of the iterates to the steady state of the underlying initial value problem. 

The analysis of the second order MPRK22($\alpha$) schemes applied to arbitrary linear systems revealed that the schemes satisfy the conditions to ensure stability and local convergence, which we confirmed with numerical experiments.

Theorem \ref{Thm_MPRK_stabil} opens up the possibility to investigate the stability and local convergence of a wide variety of numerical methods when applied to linear as well as nonlinear systems of differential equations. In particular, the analysis of  GeCo \cite{MR4087156} and BBKS \cite{BBKS2007,BRBM2008,MR4109346} schemes is still of interest. Additionally, a global stability analysis is a future research topic. 

\section{Acknowledgements}
The author Th.\ Izgin gratefully acknowledges the financial support by the Deutsche Forschungsgemeinschaft (DFG) through grant ME 1889/10-1.

\bibliographystyle{plain} 
\bibliography{cas-refs}
\end{document}